%% file: high-energy-wave_arxiv.tex
\documentclass[a4paper,11pt]{article}
\usepackage[utf8]{inputenc}
\usepackage[margin=2cm]{geometry}
\usepackage{amsmath,amssymb,amsthm,mathrsfs}

\usepackage{xcolor}
\definecolor{darkblue}{rgb}{0.15,0.2,0.65}
\definecolor{DB}{rgb}{0.3,0.3,0.3}
\definecolor{DOr}{rgb}{0.7,0.3,0.3}
\definecolor{DGr}{rgb}{0.3,0.7,0.3}
\definecolor{DBl}{rgb}{0.1,0.3,0.5}
\usepackage[hypertexnames=true, citecolor = darkblue, colorlinks = true, linkcolor = black, pdffitwindow = false, urlbordercolor = white,pdfborder={0 0 0}]{hyperref}
\usepackage{enumerate}
\usepackage{comment}

\input{macros}

\newcommand{\maketitlebefore}{\maketitle}
\newcommand{\maketitleafter}{}
\newcommand{\sectioncase}[1]{#1}

\newcommand{\doi}[1]{\textsc{doi}: \href{http://dx.doi.org/#1}{\nolinkurl{#1}}}
\excludecomment{subjclass}
\excludecomment{keywords}
\newenvironment{acknowledgements}{\section*{Acknowledgements}}{}

\newtheorem{definition}{Definition}[section]
\newtheorem{lemma}[definition]{Lemma}
\newtheorem{proposition}[definition]{Proposition}
\newtheorem{theorem}[definition]{Theorem}
\newtheorem{corollary}[definition]{Corollary}
\newtheorem{example}[definition]{Example}
\theoremstyle{remark}
\newtheorem*{remark}{Remark}

\title{High energy bounds on \moller{} operators}
\author{Henning Bostelmann\thanks{%
University of York, Department of Mathematics, York YO10 5DD, UK}
\and Daniela Cadamuro\thanks{%
Universit\"at Leipzig, Institut f\"ur Theoretische Physik, Br\"uderstra\ss{}e 16, 04103 Leipzig, Germany}
\and Gandalf Lechner\thanks{%
School of Mathematics, Cardiff University, Senghennydd Road, CF24 4AG Cardiff, UK}
}

\date{June 15, 2021}

\numberwithin{equation}{section}

\begin{document}

\input{high-energy-wave_main}

\bibliographystyle{halpha-abbrv}
\bibliography{backflow}

\end{document}

%% file: high-energy-wave_main.tex
\maketitlebefore

\begin{abstract}
The wave operators $W_\pm(H_1,H_0)$ of two selfadjoint operators $H_0$ and $H_1$ are analyzed at asymptotic spectral values. Sufficient conditions for $\|(W_\pm(H_1,H_0)-P_{1}^\mathrm{ac}P_{0}^\mathrm{ac})f(H_0)\| <\infty$ are given, where $P_{j}^\mathrm{ac}$ projects onto the subspace of absolutely continuous spectrum of $H_j$ and $f$ is an unbounded function ($f$-boundedness), both in the case of trace-class perturbations and in terms of the high-energy behaviour of the boundary values of the resolvent of $H_0$ (smooth method). Examples include $f$-boundedness for the perturbed polyharmonic operator and for Schr\"odinger operators with matrix-valued potentials. We discuss an application to the problem of quantum backflow.
\end{abstract}


\maketitleafter

\section{\sectioncase{Introduction}}

The purpose of mathematical scattering theory is to compare two selfadjoint operators $H_0$, $H_1$ acting on a common Hilbert space $\Hil$ via their {\em wave operators} (or M\o{}ller operators), defined as the strong limits
\begin{equation*}
    \mopm(H_1,H_0)
    :=
    \slim_{t\to\pm\infty}e^{itH_1}e^{-itH_0}\pac{0},
\end{equation*}
where $\pac{0}$ is the projection onto the subspace $\pac{0}\Hil$ of absolutely continuous spectrum of $H_0$. When these operators exist, they define isometries $\pac{0}\Hil\to\pac{1}\Hil$ that intertwine the absolutely continuous parts of $H_0$ and $H_1$ and can therefore be used to obtain information about $H_1$ based on information about $H_0$. This has many applications, in particular in quantum physics, where $H_0$ plays the role of a ``free'' Hamiltonian that can be investigated directly, and $H_1$ an ``interacting'', more complicated operator, that cannot be analyzed directly. The wave operators are then used to characterize asymptotic properties of the dynamics given by the unitary group $e^{-itH_1}$ in terms of the ``free'' dynamics $e^{-itH_0}$.

\sloppy
Many classical theorems in the field concentrate on establishing conditions on~$H_0$ and~$H_1$ that ensure existence of the wave operators $\mopm(H_1,H_0)$ and $\mopm(H_0,H_1)$ (see the monographs \cite{Yafaev:general, Yafaev:analytic, ReedSimon:1979} for a thorough presentation of the subject). In contrast, we are here interested in more quantitative questions at asymptotic spectral values $\lambda$ of the operators (meaning the limit $|\lambda|\to\infty$). Namely, one would expect that in typical situations $\mopm(H_1,H_0)$ approximates the identity on spectral subspaces for asymptotic spectral values. Borrowing terminology from applications to Hamiltonians, we also refer to this as the ``high energy'' behaviour. 

One indication for this is as follows: Consider the \emph{scattering operator} $S:=\mo_+(H_1,H_0)^\ast \mo_-(H_1,H_0)$, which commutes with $H_0$ and hence in a diagonalization of $H_0$ acts by multiplication with an operator-valued function $\lambda \mapsto S(\lambda)$, where $\lambda$ are the spectral values of $H_0$. In relevant examples, $S(\lambda)-\idop$ is of Hilbert-Schmidt class (``finite total cross section'' \cite{EnssSimon:crosssection}) and its Hilbert-Schmidt norm decays as $|\lambda| \to \infty$ \cite{Jensen:crosssection,SobolevYafaev:qcl}. Related results are also important in the context of the Aharanov-Bohm effect \cite{BallesterosWeder:2009}.

\fussy
In the present paper, we investigate a different and more direct question: we study the behaviour of $\mopm(H_1,H_0)-\idop$ at asymptotic spectral values of $H_1$ and~$H_0$. Since the wave operator contains information only on the absolutely continuous parts of the $H_j$, it is better to replace the identity~$\idop$ with the product of projections $\pac{1}\pac{0}$, and we define a pair $(H_1,H_0)$ of selfadjoint operators to be {\em $f$-bounded} (Def.~\ref{def:fequiv}) if their wave operators exist and satisfy
\begin{align}\label{eq:intro-fbound}
 \|(\mopm(H_1,H_0)-\pac{1}\pac{0})f(H_0)\|
 <\infty
\end{align}
for some unbounded continuous function $f:\Rl\to\Rl$ (``high energy bounds'').

This question is partially motivated by our previous work on backflow, a surprising quantum mechanical effect which describes the situation where the probability current of a quantum particle in one dimension can flow in the direction opposite to its momentum. To quantify this effect in a situation with only asymptotic (in time) information on momentum distributions, bounds of the form \eqref{eq:intro-fbound} are essential \cite{BCL:backflow}, where $H_0=-\frac{d^2}{dx^2}$ is the one-dimensional Laplacian, $H_1-H_0$ is a short range multiplication operator, and $f(\lambda)=\sqrt{|\lambda|}$.

More generally, high energy bounds are of interest whenever one wishes to quantify how small the effect of a perturbation $H_1-H_0$ of a given selfadjoint operator $H_0$ on the unitary group $e^{-itH_0}$ is at asymptotic spectral values. This is not restricted to the particular one-dimensional Schrödinger operator setup encountered in backflow, but applies to a large variety of other situations. For example, one may think of pseudodifferential operators $H_0=(-\frac{d^2}{dx^2}+m^2)^{1/2}$ to describe relativistic dynamics, integral operator potentials to describe noncommutative potential scattering \cite{DurhuusGayral:2010,LechnerVerch:2013}, or applications to systems in higher dimensions or quantum field theory.

The aim of this article is to provide a unified analysis of $f$-bounds in an abstract setting, and to provide the tools for establishing them in a wide range of examples.

In Sec.~\ref{sec:setting}, we recall some basic facts of scattering theory, and introduce the precise definition of $f$-boundedness as motivated above. We show with some a priori examples that $f$-boundedness depends crucially on the operator pair $(H_1,H_0)$: In some cases, it holds for all choices of $f$ (Example~\ref{example:allf}), whereas in other situations it holds for basically none (Example~\ref{example:nof}). Moreover, on the set of all selfadjoint operators on $\Hil$ which have bounded point spectrum and bounded singular continuous spectrum, we can strengthen mutual $f$-boundedness to an equivalence relation.

For treating examples closer to applications, there are two main methods of scattering theory, the \emph{trace-class method} (where $H_1-H_0$ is assumed to be of trace class) and the \emph{smooth method} (where additional smoothness assumptions are put on the resolvents of $H_1$ and $H_0$), and we investigate $f$-boundedness in each of these.  

We first briefly discuss the trace-class method in Sec.~\ref{sec:traceclass}, giving sufficient conditions on a trace-class perturbation $V=H_1-H_0$ for $f$-boundedness to hold (Prop.~\ref{prop:integralestimate}). Concretely, this applies in particular to rank-1 perturbations \cite{Simon:1994}.

Then, in Sec.~\ref{sec:smooth}, we discuss $f$-boundedness in the smooth method,  where specifically the choice $f_\beta(\lambda)=(1+\lambda^2)^{\beta/2}$, $\beta\in(0,1)$, turns out to be advantageous. The main technical tool is to use the limiting absorption principle and study the wave operators in terms of boundary values of the resolvents of $H_0$ and $H_1$ at their spectra, taken in a suitable norm $\|\cdot\|_{\Ban,\Band}$ defined by a Gelfand triple $\Ban\subset\Hil\subset\Band$ (see, for example, \cite{BenArtzi:smooth} for a review of this technique). We pay particular attention to deriving sufficient conditions for \eqref{eq:intro-fbound} that can be expressed exclusively in terms of ``free'' data (referring to $H_0$, its resolvent $R_0$, etc.), as required for applications. Our main result is that $f_\beta$-boundedness is essentially implied by the high-energy behaviour of the resolvent of $H_0$: if $R_0$ is locally Hölder continuous and $\| R_0(\lambda \pm i0) \|_{\Ban,\Band} = O(|\lambda|^{-\beta})$, then $f_\beta$-boundedness follows (Theorem~\ref{theorem:betaequiv}). This allows us, in particular, to treat the case where $H_0=(-\Delta)^{\ell/2}$ is a fractional power of the Laplacian: For suitable $V$, we find that $f_\beta$-boundedness holds for $(H_0+V,H_0)$ whenever $0 < \beta \leq 1-\frac{1}{\ell}$ (Example~\ref{example:laplace}), and this bound is sharp in general (Example~\ref{example:s1d}). 

We also ask whether this situation is stable under tensor product constructions, i.e., we consider $H_0 = H_A \otimes \idop + \idop \otimes H_B$ where $H_A$ is of the same type as before, and $H_B$ has only point spectrum. Under certain conditions our results transfer to this situation (Corollary~\ref{corr:tensorperturb}). In the case where $H_A$ is the negative Laplacian---in particular in applications to quantum physics---these $H_0$, and $H_0+V$ for suitable $V$, are known as \emph{matrix-valued Schr\"odinger operators} or \emph{Schr\"odinger operators with matrix-valued potentials} (see, e.g., \cite{GKM:matrixvalued,FLS:boundstates,KamRou:matrixresolvent,CJLS:matrixvalued}), although we allow the matrices to become infinite-dimensional (Example~\ref{example:dilation}). A particular problem occurs here for low-dimensional Laplacians if $H_B$ is of infinite rank; we discuss this in Example~\ref{example:tensors1}.

Finally, in Sec.~\ref{sec:qi}, we return to our motivating backflow example from quantum mechanics, and show that semiboundedness of certain operators is preserved under perturbation with a wave operator.

Appendices \ref{app:Privalov}--\ref{app:boundedkernel} recall and supply some auxiliary results needed in the main text.

\section{\sectioncase{General setting}}\label{sec:setting}

Throughout this article, our general setup will be as follows. We consider two selfadjoint operators $H_0$ and $H_1$ on a common complex separable Hilbert space $\Hil$. In our notation, we use an index $0$ or $1$ to distinguish between the spectral resolutions and subspaces related to $H_0$ or $H_1$, e.g., $\pac{0}$ is the projection onto the subspace of absolutely continuous spectrum of $H_0$, and $E_1$ is the spectral resolution of $H_1$. As the main quantity of interest, we consider the strong limits
\begin{align*}
    \mopm(H_1,H_0)
    :=
    \slim_{t\to\pm\infty}e^{itH_1}e^{-itH_0}\pac{0}
\end{align*}
and call them {\em \moller{} operators} if they exist. We will use the following well-known results about \moller{} operators (see, for example, \cite{Yafaev:general,ReedSimon:1979}), and refer to them as (W1)--(W3).

\begin{itemize}
 \item[(W1)]  $\mopm(H_1,H_0)$ are partial isometries with initial space $\Hil_{0}^{\rm ac}=\pac{0}\Hil$ and final space contained in~$\Hil_{1}^{\rm ac}$. In case their final spaces coincide with $\Hil_{1}^{\rm ac}$, the \moller{} operators are called {\em complete}. This is equivalent to the existence of $\mopm(H_0,H_1)$, and implies $\mopm(H_1,H_0)^*=\mopm(H_0,H_1)$.
 \item[(W2)]  A {\em chain rule} holds: If $H_0,H_1,H_2$ are selfadjoint such that $\mopm(H_1,H_0)$ and $\mopm(H_2,H_1)$ exist, then also $\mopm(H_2,H_0)$ exists, and 
 \begin{equation*}
  \mopm(H_2,H_0)=\mopm(H_2,H_1)\mopm(H_1,H_0).
 \end{equation*}

 \item[(W3)] The {\em intertwiner property} holds: For an arbitrary bounded Borel function $\varphi:\Rl\to\Cl$, one has
\begin{equation*}
	\varphi(H_1)\mopm(H_1,H_0)=\mopm(H_1,H_0)\varphi(H_0).
\end{equation*}
\end{itemize}

In many situations, $\mopm(H_1,H_0)$ approximates the identity, or rather the operator $\pac{1}\pac{0}$ in the presence of point spectrum, when restricted to spectral subspaces of $H_0$ for asymptotic spectral values. More specifically, one may find that $(\mopm(H_1,H_0)-\pac{1}\pac{0})f(H_0)$ is bounded despite the function $f$ being unbounded on the spectrum of $H_0$. This motivates the following definition; here and throughout the paper, $C(\Rl)$ denotes the space of continuous real-valued functions on $\Rl$.

\begin{definition}\label{def:fequiv}
 Let $H_0$ and $H_1$ be two selfadjoint operators on a Hilbert space $\Hil$ such that their \moller{} operators $\mopm(H_1,H_0)$ exist, and let $f\in C(\Rl)$. Then the pair $(H_1,H_0)$ is called \emph{$f$-bounded} if $(\mopm(H_1,H_0)-\pac{1}\pac{0})f(H_0)$ is bounded.
 If both $(H_1,H_0)$ and $(H_0,H_1)$ are $f$-bounded, then we call the operators \emph{mutually $f$-bounded}.
\end{definition}

Clearly one is interested here in functions $f$ that grow as $\lambda \to \pm \infty$; a typical choice would be $f_\beta(\lambda):=(1+\lambda^2)^{\beta/2}$, $\beta>0$. 
This raises the question which rate of growth of $f(\lambda)$ as $\lambda\to\pm\infty$ is still compatible with $f$-boundedness, to be investigated in later sections.

Heuristically, $f$-boundedness should be determined by the behaviour of the \moller{} operator at asymptotic spectral values of $H_0$ {\em and} $H_1$. In the following lemma, we show that under the additional restriction $f(H_1)-f(H_0)\in\boundedops$, this can in fact be made precise. This condition holds for a large class of $f$ in case $H_1-H_0$ is bounded, see Lemma~\ref{lemma:fractional} in Appendix~B.

\begin{lemma}\label{lemma:hboundequiv}
  Let $H_0$, $H_1$ be two selfadjoint operators, and $f\in C(\Rl)$. Suppose that  $\mo :=\mopm(H_1,H_0)$ exists and is complete,
   and that $f(H_1)-f(H_0)$ is bounded. Then the following statements are equivalent:
  \begin{enumerate}[(i)]
   \item \label{it:h0} $\big(\mo -\pac{1}\pac{0}\big) f(H_0) \in \boundedops$;
   \item \label{it:h0some} $E_1(-\lambda,\lambda)^\perp \big(\mo -\pac{1}\pac{0}\big) f(H_0) E_0(-\lambda,\lambda)^\perp \in \boundedops$ for some $\lambda>0$;
   \item \label{it:h0all} $E_1(-\lambda,\lambda)^\perp \big(\mo -\pac{1}\pac{0}\big) f(H_0) E_0(-\lambda,\lambda)^\perp \in \boundedops$ for all $\lambda>0$.
  \end{enumerate}
Here $E_j(-\lambda,\lambda)^\perp:= \idop -E_j(-\lambda,\lambda)$.
\end{lemma}

\begin{proof}
 It is clear that   (\ref{it:h0}) $\Rightarrow$ (\ref{it:h0all}) $\Rightarrow$  (\ref{it:h0some}).  
 For (\ref{it:h0some}) $\Rightarrow$ (\ref{it:h0}), since $f(H_0) E_0(-\lambda,\lambda)$ is bounded, it only remains to show that 
 $
      E_1(-\lambda,\lambda) \big(\mo -\pac{1}\pac{0}\big) f(H_0) \in \boundedops.
 $
 To that end, we consider a compact set $\Delta\subset\Rl$ and define the bounded Borel function $f_\Delta:=f\cdot\chi_\Delta$. Using the intertwining property (W3) of~$\mo$, we get
 \begin{multline*}
      E_1(-\lambda,\lambda) \big(\mo -\pac{1}\pac{0}\big) f_\Delta(H_0) = \\ E_1(-\lambda,\lambda) f_\Delta(H_1) (\mo - \pac{1}\pac{0}) + E_1(-\lambda,\lambda)\pac{1} \big(f_\Delta(H_1)-f_\Delta(H_0)\big)\pac{0}.    
 \end{multline*}
 As $\Delta\to\Rl$, this operator remains bounded by assumption.
\end{proof}

As an aside, we mention that in the situation of two mutually $f$-bounded operators $H_0$, $H_1$, one has $\mopm(H_0,H_1)^*=\mopm(H_1,H_0)$ by (W1), with which it is easy to show that $\pac{1} (f(H_1)-f(H_0))\pac{0}$ is bounded. This shows that the assumption used in Lemma~\ref{lemma:hboundequiv} is quite natural.

Under some extra restrictions on the selfadjoint operators involved, we can strengthen (mutual) $f$-boundedness to an equivalence relation.

\begin{definition}\label{def:equiv}
Let $\aac(\Hil)$  (the ``almost absolutely continuous'' operators) be the set of densely defined selfadjoint operators $H$ on the space $\Hil$ such that $(\idop-\pac{})H$ is bounded.

Fix $f \in C(\Rl)$. For $H_0,H_1 \in \aac(\Hil)$, we say that $H_1$ and $H_0$ are \emph{$f$-equivalent} (written as $H_1 \sim_f H_0$) if $(H_1,H_0)$ is $f$-bounded, $\mopm(H_1,H_0)$ are complete, and
$f(H_0)-f(H_1)$ is bounded.
\end{definition}

The name ``$f$-equivalent'' is justified because of
\begin{proposition} \label{prop:fequiv}
Let $f \in C(\Rl)$.

\begin{enumerate}[(i)]
 \item \label{it:onebound} If $H_1,H_0\in\aac(\Hil)$ such that $H_1 \sim_f H_0$, then $(\mopm(H_1,H_0)-\idop)f(H_0)$ are bounded.
 \item \label{it:equiv} $\sim_f$ is an equivalence relation on $\aac(\Hil)$. 
\end{enumerate}
 
\end{proposition}

\begin{proof}
 For part (\ref{it:onebound}), one notes that
 \begin{equation}\label{eq:onepdiff}
 \begin{aligned}
    (\idop - \pac{1}\pac{0})f(H_0) = & \pac{1}(\idop - \pac{0})f(H_0) + (\idop - \pac{1})f(H_1) \\ &+ (\idop - \pac{1})(f(H_0)-f(H_1)),
 \end{aligned}
 \end{equation}
 and all terms on the right-hand side are bounded by assumption. 
 
 In (\ref{it:equiv}), reflexivity is evident. For symmetry, let $H_1 \sim_f H_0$. We cut down $f$ to $f_\Delta=f\cdot\chi_\Delta$ with a compact $\Delta\subset\Rl$. (W1) and the intertwining property (W3) then yield
\begin{equation*}
\begin{aligned}
  \big(\mopm(H_0,H_1) &-\idop\big) f_\Delta(H_1)
 = \Big( f_\Delta(H_1) \big(\mopm(H_1,H_0) -\idop\big)\Big)^\ast \\
 & =\Big(  \big(\mopm(H_1,H_0) -\idop\big)  f_\Delta(H_0) \Big)^\ast + \big(f_\Delta(H_0) - f_\Delta(H_1)\big).
\end{aligned}
\end{equation*}
As $\Delta\to\Rl$, this expression remains bounded to part (\ref{it:onebound}), and we obtain $H_0 \sim_f H_1$ using \eqref{eq:onepdiff}. Likewise, transitivity follows using the chain rule (W2).
\end{proof}

A pair of selfadjoint operators $H_0$, $H_1$ that have \moller{} operators can exhibit very different behaviour regarding $f$-boundedness, as we now demonstrate with two examples: The first example shows that $f$-boundedness can fail even for functions $f$ of arbitrarily slow divergence rate at $+\infty$, whereas the second example shows that $f$-boundedness can hold for every $f$.

\begin{example}\label{example:nof}
    Let $\Hil=L^2(\Rl,d x)$ and $H_0:=-i\frac{d}{d x}$. Let $v\in L^1(\Rl)\cap L^\infty(\Rl)$ be non-zero and real, and define $H_1:=H_0+V$, where $V$ is the operator multiplying with $v$. Then the \moller{} operators $\mopm:=\mopm(H_0,H_1)$ and $\mopm(H_1,H_0)$ exist. If $f\in C(\Rl)$ is such that $\Rl_+\ni\lambda\mapsto|f(\lambda)|$ is monotonically increasing to $+\infty$, then $(H_1,H_0)$ is not $f$-bounded.
\end{example}
\begin{proof}
    Note that $H_0$ and $H_1$ are selfadjoint on their natural domains. As shown in \cite[p.~83-84]{Yafaev:general}, $H_0$ and $H_1$ are unitarily equivalent and have absolutely continuous simple spectrum covering the full real axis. Furthermore, the \moller{} operators $\mopm:=\mopm(H_0,H_1)$ exist and are unitary. Explicitly, they act as 
    \begin{align*}
            (\mopm\psi)(x)=w_\pm(x)\psi(x)\,,\qquad 	w_\pm(x):=\exp i\int_x^{\pm\infty}v(y)\,d y.
    \end{align*}
    \sloppy
    This implies in particular that the \moller{} operators are complete, and that $\mopm(H_1,H_0)=\mopm^*$ acts by multiplication with $\overline{w_\pm}$.
    
    We have to show that $(\mopm-\idop)f(H_0)=(f(H_0)(W_\mp-\idop))^*$ is unbounded. But in view of the form of $\mopm$, we find an interval $I$ such that the restricted operator $(\mopm-\idop):L^2(I)\to L^2(I)$ has a bounded inverse. Hence we only have to show that $f(H_0)$ is unbounded on $L^2(I)$ and for this purpose may assume $I=[0,r]$ for some $r>0$ because $H_0$ commutes with translations. Let $\psi\in L^2(I)$ be the normalized characteristic function of $I$, and consider the normalized sequence $\psi_n(x):=\sqrt{n}\psi(nx)$ in $L^2(I)$.
    Then we find by straightforward calculations and estimates
    \begin{equation*}
    \begin{aligned}
		\|f(H_0)\psi_n\|^2
		&=
		\int_\Rl |f(nk)|^2\,|\tilde\psi(k)|^2dk
		\\
		&\geq
		\int_1^\infty |f(nk)|^2\,|\tilde\psi(k)|^2dk
		\geq
		|f(n)|^2\int_1^\infty |\tilde\psi(k)|^2dk,
    \end{aligned}
    \end{equation*}
    where we have used the monotonicity assumption on $f$. This property also shows that the last expression goes to infinity as $n\to\infty$, which implies the claim.
\end{proof}

\begin{example}\label{example:allf}
    Let $H_0$ be a selfadjoint operator with purely absolutely continuous spectrum, $V$ a selfadjoint bounded operator on the same Hilbert space $\Hil$, and $H_1:=H_0+V$. Assume that the \moller{} operators $\mopm(H_1,H_0)$ exist and that there exists a compact $\Delta\subset\sigma(H_0)$ such that $V\Hil\subset E_0(\Delta)\Hil$. Then $(H_1,H_0)$ is $f$-bounded for every $f\in C(\Rl)$.
\end{example} 
\begin{proof}
    With $E_0^\perp:=\idop-E_0(\Delta)$, our assumption can be rephrased as $0=VE_0^\perp=(H_1-H_0)E_0^\perp$, which implies $H_1E_0^\perp=H_0E_0^\perp=E_0^\perp H_0$ on $\dom H_0=\dom H_1$, and, by selfadjointness, also $E_0^\perp H_1=E_0^\perp H_0$ on this domain. A power series calculation on analytic vectors for $H_0$ then gives $E_0^\perp e^{-itH_0}=e^{-itH_1}E_0^\perp$, and therefore 
    \begin{equation*}
\mopm(H_1,H_0)E_0^\perp =\slim_{t\to\pm\infty} e^{itH_1}e^{-itH_0}E_0^\perp=\slim_{t\to\pm\infty}e^{itH_1}E_0^\perp e^{-itH_0}=E_0^\perp.     
    \end{equation*}
 Thus $(\mopm(H_1,H_0)-\pac{1})=\pac{1}(\mopm(H_1,H_0)-\idop)=\pac{1}(\mopm(H_1,H_0)-\idop)E_0(\Delta)$, and for arbitrary $f\in C(\Rl)$, we have the bound 
    \begin{equation*}
    \begin{aligned}
     \|(\mopm(H_1,H_0)-\pac{1})f(H_0)\|
     &=
     \|(\mopm(H_1,H_0)-\pac{1})f_\Delta(H_0)\|
     \\ &\leq 
     2\|f_\Delta\|_\infty <\infty. \qed
    \end{aligned}
    \end{equation*}
    \noqed
\end{proof}

A concrete realisation of this situation on $\Hil=L^2(\Rl,d x)$ is given by $H_0=-\frac{d^2}{d x^2}$ and $V$ an integral operator such that the Fourier transform of its kernel lies in $C_0^\infty(\Rl\times\Rl)$. Then $V$ is trace-class, which implies existence and completeness of the \moller{} operators by the Kato-Rosenblum theorem \cite[Thm.~XI.8]{ReedSimon:1979}, and the assumption $V\Hil\subset E_0(\Delta)\Hil$ follows from the support of the integral kernel.

\section{\sectioncase{Trace class perturbations and} \texorpdfstring{$f$}{f}-\sectioncase{boundedness}}\label{sec:traceclass}

The Kato-Rosenblum theorem states that if $H_0$ and $H_1$ are selfadjoint and their difference is trace class, then $\mopm(H_1,H_0)$ exist and are complete. Examples of such trace class perturbations are given by integral operators with suitable kernels, or rank one perturbations as the simplest case. 

We now investigate $f$-boundedness in this setting and first recall some relevant notions. For a selfadjoint operator $H_0$ on $\Hil$ and a vector $\xi\in\pac{0}\Hil$, we define
\begin{align*}
 \|\xi\|_{H_0}^2
 :=
 \esssup_{\lambda\in\hat\sigma(H_0)}\left|\frac{d\langle\xi,E_0(\lambda)\xi\rangle}{d\lambda}\right|
 =
 \esssup_{\lambda\in\hat\sigma(H_0)}\left\|\xi_\lambda\right\|_{\lambda}^2 
 \in[0,+\infty],
\end{align*}
where $\hat\sigma(H_0)$ is the core of the spectrum $\sigma(H_0)$. The second expression refers to the direct integral decomposition of the absolutely continuous subspace,
\begin{align*}
 \pac{0}\Hil = \int^\oplus_{\hat\sigma(H_0)}\mathfrak{h}(\lambda) d\lambda,
\end{align*}
namely $\xi_\lambda$ is the component of $\xi\in\pac{0}\Hil$ in $\mathfrak{h}(\lambda)$, and $\|\cdot\|_\lambda$ is the norm of $\mathfrak{h}(\lambda)$ \cite[p.~32]{Yafaev:general}. The set of all $\xi$ with finite $\|\xi\|_{H_0}$ is $\|\cdot\|$-dense in $\pac{0}\Hil$, and $\|\cdot\|_{H_0}$ is a norm on it \cite{ReedSimon:1979}.

Note that if $H_0,H_1$ are two selfadjoint operators with complete \moller{} operators, then $\mopm(H_0,H_1):\pac{1}\Hil\to\pac{0}\Hil$ are unitaries intertwining the absolutely continuous parts of $H_0$ and $H_1$. This implies that we can identify the direct integral decompositions of $\pac{0}\Hil$ and $\pac{1}\Hil$, and $\|(\mopm(H_0,H_1)\xi)_\lambda\|_\lambda=\|\xi_\lambda\|_\lambda$ for each $\lambda\in\hat\sigma(H_0)$. In particular, $\|\mopm(H_0,H_1)\xi\|_{H_0}=\|\xi\|_{H_1}$ in this situation.

The following lemma due to Rosenblum \cite[Lemma~1, p.23]{ReedSimon:1979} will be essential in our discussion.

\begin{lemma}\label{lemma:rosenblum}
    Let $H_0$ be a selfadjoint operator on $\Hil$ and $\psi,\xi$ vectors in $\Hil$, with $\xi\in\pac{0}\Hil$ such that $\|\xi\|_{H_0}<\infty$. Then
    \begin{align*}
        \int_{-\infty}^\infty |\langle\psi,e^{\pm itH_0}\xi\rangle|^2d t\leq 2\pi\|\xi\|_{H_0}^2\|\psi\|^2<\infty.
    \end{align*}
\end{lemma}

As in the Kato-Rosenblum theorem, we now consider two selfadjoint operators $H_0,H_1$ such that $V:=H_1-H_0$ is trace class, and hence has an expansion $V\psi=\sum_n t_n \langle\xi_n,\psi\rangle\xi_n$, $\psi\in\Hil$, where the $\xi_n$ form an orthonormal basis of $\Hil$ and the sequence $(t_n)$ is summable, $\sum_n|t_n|<\infty$.

The idea of the $f$-boundedness result presented below is to choose $V$ such that not only $Vf(H_0)$ is trace-class, but also 
\begin{align*}
 \|Vf(H_0)\|_{H_1,H_0}^\Lambda
 :=
 \sum_n |t_n|\|\pac{1}E_1(-\Lambda,\Lambda)^\perp\xi_n\|_{H_1}\|\pac{0}E_0(-\Lambda,\Lambda)^\perp f(H_0)\xi_n\|_{H_0}
\end{align*}
is finite, where $\Lambda\geq0$ is arbitrary.

\begin{proposition}\label{prop:integralestimate}
 Let $H_0,H_1$ be selfadjoint with $V:=H_1-H_0$ of trace class and $\|Vf(H_0)\|_{H_1,H_0}^\Lambda<\infty$ for some $\Lambda\geq0$. Let $f\in C(\Rl)$ such that $f(H_1)-f(H_0)$ is bounded. Then $\mopm(H_1,H_0)$ exist, are complete, and satisfy
 \begin{align*}
  \|(\mopm(H_1,H_0)-\pac{1}\pac{0})f(H_0)\| \leq 2\pi\|Vf(H_0)\|_{H_1,H_0}^\Lambda < \infty.
 \end{align*}
 In particular, $(H_1,H_0)$ is $f$-bounded.
\end{proposition}
\begin{proof}
   Existence and completeness of $\mo:=W_\mp(H_1,H_0)$ follow from the Kato-Rosenblum Theorem. To obtain $f$-boundedness, we first recall that for $\varphi\in\dom H_1$, $\psi\in\dom H_0$, we have 
   \begin{align*}
    \frac{d}{dt}\langle \varphi,\pac{1}e^{\mp itH_1}e^{\pm i tH_0}\pac{0}\psi\rangle
    =
    \mp i\langle\varphi,\pac{1}e^{\mp itH_1}Ve^{\pm i tH_0}\pac{0}\psi\rangle.
   \end{align*}
	After integration, this yields
	\begin{align*}
		\langle\varphi,(\mo-\pac{1}\pac{0})\psi\rangle
		=
		\pm i\int_0^\infty \langle\pac{1}\varphi,e^{\mp itH_1}Ve^{\pm i tH_0}\pac{0}\psi\rangle dt.
	\end{align*}
   The proof is based on this identity, the expansion $V\psi=\sum_n t_n \langle\xi_n,\psi\rangle\xi_n$ of $V$, and Lemma~\ref{lemma:rosenblum}. With $E_j^\Lambda:=E_j(-\Lambda,\Lambda)^\perp$, we find
    \begin{align*}
        &\bigg|\langle  E_1^\Lambda\varphi,(\mo-\pac{1}\pac{0})f(H_0)E_0^\Lambda\psi\rangle\bigg|
        \\
        &=
        \left|\int_0^\infty \langle\varphi,e^{\mp itH_1}\pac{1}E_1^\Lambda VE_0^\Lambda e^{\pm i tH_0}f(H_0)\pac{0}\psi\rangle dt\right|
        \\
        &\leq
        \sum_n|t_n|
        \left|\int_0^\infty \langle \varphi,e^{\mp itH_1}\pac{1}E_1^\Lambda\xi_n\rangle
        \langle E_0^\Lambda\xi_n,e^{\pm i tH_0}f(H_0)\pac{0}\psi\rangle dt\right|
        \\
        &\leq
        \sum_n|t_n|\left( \int_0^\infty |\langle\pac{1}E_1^\Lambda\xi_n, e^{\pm itH_1}\varphi\rangle|^2 dt
        \cdot 
        \int_0^\infty |\langle f(H_0)\pac{0}E_0^\Lambda\xi_n,e^{\pm i tH_0}\pac{0}\psi\rangle|^2 dt\right)^{\tfrac{1}{2}} \!\!.
    \end{align*}
    We can now use Lemma~\ref{lemma:rosenblum} to estimate the integrals, and arrive at the bound
    \begin{equation*}
    \begin{aligned}
		\big\lvert \langle\varphi, & E_1^\Lambda(\mo-\pac{1}\pac{0})  f(H_0)E_0^\Lambda\psi\rangle \big\rvert
		\\
		&\leq
		2\pi\,\sum_n|t_n|\|\pac{1}E_1^\Lambda\xi_n\|_{H_1}\|f(H_0)E_0^\Lambda\pac{0}\xi_n\|_{H_0}\,\|\varphi\|\|\psi\|
        \\
        &=
        2\pi\,\|Vf(H_0)\|_{H_1,H_0}^\Lambda\,\|\varphi\|\|\psi\|.
    \end{aligned}   
    \end{equation*}
    In view of Lemma~\ref{lemma:hboundequiv}, this finishes the proof.
\end{proof}

We note that the assumption $f(H_1)-f(H_0)$ being bounded was only used for the reference to Lemma~\ref{lemma:hboundequiv} and can therefore be dropped for $\Lambda=0$.

Particular examples of trace-class perturbations that have attracted continued attention are perturbations by rank one operators $V=\langle\xi,\,\cdot\,\rangle\,\xi$ with some $\xi\in\Hil$ \cite{Simon:1994}. In that case, the bound from the previous proposition is
\begin{equation*}
2\pi\|\pac{1}E_1(-\Lambda,\Lambda)^\perp\xi\|_{H_1}\|\pac{0}E_0(-\Lambda,\Lambda)^\perp f(H_0)\xi\|_{H_0}, 
\end{equation*}
 and in order to have it finite, we need to control the ``spectral norms'' of both $H_0$ and $H_1$. Whereas the norm coming from $H_0$ can typically be controlled directly in applications, this is typically not the case for $H_1$. Let us therefore clarify how an estimate on
 \begin{equation*}
\|\pac{1}E_1(-\Lambda,\Lambda)^\perp\xi\|_{H_1}=\esssup_{|\lambda|\geq\Lambda}\left| \frac{d\langle\xi,E_1(\lambda)\xi\rangle}{d\lambda}\right|   
 \end{equation*}
can be given in terms of $H_0$:
With the resolvents $R_0,R_1$ of $H_0$, $H_1$, one has 
\begin{align*}
 \frac{d\langle\xi,E_j(\lambda)\xi\rangle}{d\lambda}
 =
 \frac{1}{\pi}\im\langle\xi,R_j(\lambda+i0)\xi\rangle,
\end{align*}
\sloppy
and for a rank one perturbation $H_1=H_0+\langle\xi,\,\cdot\,\rangle\,\xi$, one moreover has the Aronszajn-Krein formula \cite{Simon:1994}
\begin{align*}
 \langle\xi,R_1(\lambda+i0)\xi\rangle=\frac{\langle\xi,R_0(\lambda+i0)\xi\rangle}{1+\langle\xi,R_0(\lambda+i0)\xi\rangle}
 .
\end{align*}
It therefore follows that in case $\xi$ is such that $\|\xi\|_{H_0}<\infty$ and the boundary values $\langle\xi,R_0(\lambda+i0)\xi\rangle$ converge to $0$ as $|\lambda|\to\infty$, then also $\|\xi\|_{H_1}<\infty$.

\fussy
As a concrete example, we may take $\Hil=L^2(\Rl,dx)$ with $H_0=-\frac{d^2}{dx^2}$ and $H_1=H_0+\langle\xi,\,\cdot\,\rangle\,\xi$, where $\xi\in\mathscr{S}(\Rl)$ is a Schwartz function. Then the spectral measure of $P=i\frac{d}{d x}$ is given by $d\langle\xi,E_P(p)\xi\rangle=|\tilde\xi(p)|^2d p$ with $\tilde\xi$ the Fourier transform of $\xi\in\Hil$. After substituting $\lambda=p^2$, this shows that the spectral measure of $H_0=P^2$ is $d\langle\xi,E_0(\lambda)\xi\rangle=\frac{1}{2}\lambda^{-1/2}(|\tilde\xi(\sqrt{\lambda})|^2+|\tilde\xi(-\sqrt{\lambda})|^2)d \lambda$. Hence $\|E_0(-\Lambda,\Lambda)^\perp\xi\|_{H_0}<\infty$ and $\|E_0(-\Lambda,\Lambda)^\perp f(H_0)\xi\|_{H_0}<\infty$ for any $\xi\in\mathscr{S}(\Rl)$, $\Lambda>0$ and polynomially bounded $f\in C(\Rl)$. It is also well known that $\langle\xi,R_0(\lambda\pm i0)\xi\rangle\to0$ as $|\lambda|\to\infty$ (see Example~\ref{example:laplace} below). Hence $(H_1,H_0)$ is $f$-bounded in this situation.

\section{\sectioncase{Smooth method and} \texorpdfstring{$f$}{f}-\sectioncase{boundedness}}\label{sec:smooth}

We now discuss another specific setting of scattering theory, known as the \emph{smooth method}, which is applicable in particular to cases where one of the operators $H_j$ is a (pseudo)differential operator.

The idea behind this is as follows. If $H$ is a selfadjoint operator and $R(z)=(H-z)^{-1}$ its resolvents, then the operator-valued function $z \mapsto R(z)$ is certainly analytic on the open half planes $\mathbb{H}_\pm := \rbb\pm i \rbb_+$. One now demands that, in a suitable topology, it extends to the boundaries of the half planes, and that the extended functions are locally H\"older continuous on $\overline{ \mathbb{H} }_\pm$ (possibly with the exception of a null set). In this case $H$ is called \emph{smooth}. If both $H_0$ and $H_1$ are smooth, then the \moller{} operators $\mopm(H_1,H_0)$ automatically exist and are complete. In fact, one can express the \moller{} operators, as well as other relevant quantities, in terms of the boundary values of the resolvents, $R_1(\lambda \pm i0)$ and $R_0(\lambda \pm i0)$. We recall the basic facts about this setting in Sec.~\ref{sec:smoothsetting}, mainly in the spirit of \cite{KatoKuroda:scattering,BenArtzi:smooth}; see also \cite{Yafaev:general}.

In the context of this setting, we are interested in mutual $f$-boundedness of two operators $H_1$ and $H_0$, where we restrict to $f_\beta(\lambda)=(1+\lambda^2)^{\beta/2}$ with some $\beta\in(0,1)$. It turns out that this is implied by the behaviour of $R_0(\lambda \pm i0)$ at large $\lambda$ alone: If a certain norm of this operator is $O(|\lambda|^{-\beta})$, then mutual $f_\beta$-boundedness or even $f_\beta$-equivalence follows (Theorem~\ref{theorem:betaequiv}). 

We apply this result to examples of pseudodifferential operators (Sec.~\ref{sec:doExamples}) and investigate stability under tensor product constructions (Sec.~\ref{sec:tensor}).

\subsection{Setting of the smooth method} \label{sec:smoothsetting}

Let $\Ban$ be a Banach space
which is continuously and densely embedded in $\Hil$.
The scalar product $\hscalar{\cdotarg}{\cdotarg}$ on $\Hil$ then provides an embedding $\Hil \subset \Band$, $\varphi \mapsto \hscalar{\cdotarg}{\varphi}$, where $\Band$ denotes the conjugate dual of $\Ban$,
yielding a so-called Gelfand triple $\Ban\subset\Hil\subset\Band$.
We assume that the embedding $\Hil \subset \Band$ is dense,
so that $\hscalar{\cdotarg}{\cdotarg}$ extends to a dual pairing between $\Ban$ and $\Band$, which we denote by the same symbol.
(In most applications, $\Ban$ is actually a Hilbert space, but we stress that in this case, the dual pairing $\hscalar{\cdotarg}{\cdotarg}$ is normally \emph{not} the scalar product on $\Ban$; the latter plays no role in our investigation.)
In this setting, let us define the class of operators $H$ of interest. 

\begin{definition}\label{def:smooth}
Let $H$ be a selfadjoint operator on a dense domain in $\Hil$ and $R(z)$ its resolvents. We call $H$ an \emph{$\Ban$-smooth} operator if there exists an open set $U \subset \rbb$ of full (Lebesgue) measure such that the limits
 \begin{equation*}
     R(\lambda \pm i 0) := \lim_{\epsilon \downarrow 0} R(\lambda \pm i \epsilon), \quad \lambda \in U,
 \end{equation*}
 exist in $\bops(\Ban,\Band)$, and the extended functions $R: \mathbb{H}_\pm \cup U \to \bops(\Ban,\Band)$ are locally H\"older continuous.
 \end{definition}

In this situation, it follows that $U\subset\sigma_{\rm ac}(H)$ and for any Borel set $\Delta\subset U$, one has $E(\Delta)\Hil\subset\pac{}\Hil$. The locally H\"older continuous map $A:U \to \bops(\Ban,\Band)$ given by
\begin{equation}\label{eq:rdiff}
   A(\lambda) = \frac{1}{2\pi i} \big( R(\lambda+i0) - R(\lambda -i0) \big)
\end{equation}
equals the weak derivative $\frac{dE}{d\lambda}$ where $E$ are the spectral projections of $H$, i.e.,
\begin{equation}\label{eq:derivA}
     \frac{d}{d\lambda} \hscalar{\varphi}{E(\lambda) \psi} =  \hscalar{\varphi}{A(\lambda)\psi} \quad \text{for all $\varphi,\psi \in \Ban$, $\lambda \in U$}.
\end{equation}
It follows that $A(\lambda)$ diagonalizes the absolutely continuous part of $H$, in the sense that for bounded Borel functions $f$,
\begin{equation}\label{eq:spectralint}
   \hscalar{\varphi}{f(H)\pac{} \psi}  = \int d\lambda\, f(\lambda) \,\hscalar{\varphi}{A(\lambda)\psi} \quad \text{for all } \varphi,\psi \in\Ban.
\end{equation}
Also, we note that the $A(\lambda)$ are positive as quadratic forms on $\Ban \times \Ban$, and therefore one has 
\begin{equation}\label{eq:albound}
  \forall \varphi\in \Ban: \quad \| A(\lambda) \varphi \|^2_{\Band} \leq \|A(\lambda)\|_{\Ban,\Band} \hscalar{\varphi}{A(\lambda) \varphi}.
\end{equation}
Equally, one can start with the maps $A(\lambda)=\frac{d}{d\lambda} E(\lambda)$ and deduce the properties of the resolvents $R(\lambda\pm i0)$, see \cite[Sec.~3]{BenArtzi:smooth}. We give the following sufficient criterion:
\begin{lemma}\label{lemma:ator}
  Let $A(\lambda)=\frac{dE}{d\lambda}\in\bops(\Ban,\Band)$ be locally H\"older continuous in $\lambda\in U$, where $U\subset \rbb$ is an open set of full measure, and suppose in addition that the function $\lambda \mapsto (1+|\lambda|)^{-1} \norm{A(\lambda)}_{\Ban,\Band} $ is integrable over $\rbb$. Then $H$ is $\Ban$-smooth, and \eqref{eq:rdiff} holds.
\end{lemma}

To see this, one considers the integral $R(z)=\int A(\lambda)(\lambda-z)^{-1}d\lambda$ weakly on $\Ban \times \Ban$, where $\im z \neq 0$, splits the integration region into a small interval $J$ around $\re z \in U$ and into its complement, then applies the Privalov-Korn theorem (Lemma~\ref{lemma:Privalovop}) to the integral over $J$, and the integrability condition on $\rbb \backslash J$. See \cite[Thm.~3.6]{BenArtzi:smooth} for details. We will establish a quantitative version of this result below, in Proposition~\ref{prop:ahighenergy}.

Passing to the setting of scattering theory, let 
$\Gamma_2(\Band,\Ban)\subset\bops(\Band,\Ban)$ be the space of bounded operators that ``factor through a Hilbert space''; that is, $V\in\Gamma_2(\Band,\Ban)$ is of the form $V=V_1\st V_0$ where $V_0,V_1\in\bops(\Band,\kcal)$ with a Hilbert space $\kcal$. 
If $\Ban$ is Hilbertisable, then $\Gamma_2(\Band,\Ban)=\bops(\Band,\Ban)$; but in general, the inclusion may be proper \cite{Pisier:factorization}.

\begin{definition}
A \emph{smooth scattering system}  $(H_0,H_1,\Ban,\Hil)$ consists of two selfadjoint operators  $H_0$ and $H_1$ on a common dense domain in the Hilbert space $\Hil$ which are \emph{both} $\Ban$-smooth, such that $V:=H_1-H_0\in\Gamma_2(\Band,\Ban)$.  
\end{definition}

We can (and often will) assume in this situation that both $H_j$ are smooth with respect to the same set $U$ of full measure.
In practical examples, $\Ban$-smoothness of one operator (say, $H_0$) can usually be verified directly, whereas the $\Ban$-smoothness of $H_1$ is obtained by perturbation arguments. We give a well-known type of sufficient criterion (cf. \cite{Yafaev:general,BenArtzi:smooth}), and sketch its proof in our context.

\begin{lemma}\label{lemma:perturb}
 Let $H_0$ be a selfadjoint operator on a dense set $\dcal \subset \Hil$ and $V=V^\ast \in \bops(\Band,\Ban)$. Suppose that $H_0$ is $\Ban$-smooth and that
 $R_0(z)\in\mathrm{FA}(\Ban,\Band)$ for all $z \in \cbb \backslash \rbb$.
 Then, $H_1:=H_0+V$ is $\Ban$-smooth.
\end{lemma}
(Here $\mathrm{FA}(\Ban,\Band)\subset\bops(\Band,\Ban)$ denotes the norm closure of the space of finite rank operators. If $\Ban$ is a Hilbert space, $\mathrm{FA}(\Ban,\Band)$ equals the space of compact operators.)
\begin{proof}
 Since $V\upharpoonright \Hil = (V\upharpoonright \Hil)^\ast \in\bops(\Hil)$, also $H_1$ is selfadjoint on $\dcal$. Now for any $z \in \mathbb{H}_\pm $, the operator $\idop + VR_0(z)$ is invertible in $\bops(\Ban)$. (Otherwise, since $VR_0(z)\in \mathrm{FA}(\Ban,\Ban)$, the Fredholm alternative yields a $\psi \in\Ban\backslash \{0\}$ in the kernel of $\idop + VR_0(z)$. Then $\varphi:=R_0(z)\psi\in\dcal$ fulfills $(H_0+V)\varphi = z\varphi$ with $\im z \neq 0$, contradicting the selfadjointness of $H_1$.) By analytic Fredholm theory \cite[Theorem~3.14.3]{Simon:analysis4}, the $\bops(\Ban)$-valued function $F:z\mapsto (\idop+VR_0(z))^{-1}$ is analytic in $\mathbb{H}_\pm$. Further, for $\lambda\in U$, the norm limit $VR_0(\lambda\pm i0)$ lies in $\mathrm{FA}(\Ban,\Ban)$, hence $F(\lambda\pm i0):=(\idop+VR_0(\lambda\pm i0))^{-1}$ exists for $\lambda$ outside a closed null set $N$, cf.~\cite[Sec.~1.8.3]{Yafaev:general}; and (local H\"older) continuity of $VR_0(\cdot)$ on $\mathbb{H}_\pm \cup U$ translates to (local H\"older) continuity of $F$ there. Finally, the resolvent identity 
 \begin{equation} \label{eq:resident} 
  R_1(z) = R_0(z) - R_1(z) V R_0(z)
 \end{equation}
implies $R_1(z) = R_0(z)F(z)$, showing that $H_1$ is $\Ban$-smooth with respect to $U \backslash N$ rather than $U$.
\end{proof}

Now for a smooth scattering system, the \moller{} operators are known to exist automatically, and they can be expressed in terms of boundary values of resolvents. We derive a related expansion for the operators $\pac{1}(W_\pm(H_1,H_0)-\idop)\pac{0}$ which are of interest to us.

\begin{proposition}\label{prop:moellersmooth}
 Let $(H_0,H_1,\Ban,\Hil)$ be a smooth scattering system. Then the \moller{} operators $\mopm(H_1,H_0)$ exist and are complete. We have for real H\"older continuous functions $\chi_0,\chi_1$ with support in compact intervals $I_0,I_1 \subset U$ and for any $\varphi_0,\varphi_1 \in\Ban$, 
\begin{equation}\label{eq:mollerexpand}
\begin{aligned}
  \bighscalar{\chi_1(H_1) \varphi_1 }{ & \big(\mopm(H_1,H_0) -\idop\big) \chi_0(H_0) \varphi_0} 
  \\ &= \lim_{\epsilon\downarrow 0} \int d\lambda \,d\mu \frac{\chi_1(\mu)\chi_0(\lambda)}{\lambda-\mu \mp i\epsilon} \hscalar{\varphi_1}{ A_1(\mu) V A_0(\lambda) \varphi_0},
  \end{aligned}
\end{equation}
where $V=H_1-H_0$.
\end{proposition}

\begin{proof}
 The existence of \moller{} operators follows from well-known results: If $I \subset U$ is a compact interval, and $V=V_1\st V_0$, then $\xcal$-smoothness of $H_j$ implies that $V_j (R_j(\lambda+i\epsilon) - R_j(\lambda -i\epsilon)) V_j^\ast$ is uniformly bounded in $\mathfrak{B}(\kcal)$ for $\lambda \in I$ and $|\epsilon|$ sufficiently small. 
 Hence the operators $V_jE_j(I)$ are Kato-smooth with respect to $H_j$ \cite[Theorem~4.3.10]{Yafaev:general}. This suffices to show that the \moller{} operators $\mopm(H_1,H_0)$ and $\mopm(H_0,H_1)$ exist, since $U$ is of full measure \cite[Corollary~4.5.7]{Yafaev:general}.
 
 Given the existence of the wave operators, they have the ``stationary'' representations \cite[Lemma 2.7.1]{Yafaev:general}
 \begin{equation} \label{eq:mostatic}
     \hscalar{\psi_1}{\mopm(H_1,H_0)\psi_0} = \lim_{\epsilon \downarrow 0}\frac{\epsilon}{\pi} \int 
     \hscalar{\psi_1}{R_1(\lambda \mp i \epsilon) R_0(\lambda \pm i \epsilon)\psi_0} \; d  \lambda
 \end{equation}
 for all $\psi_j \in \pac{j}\Hil$.
 On the other hand, for every $\epsilon>0$,
 \begin{equation} \label{eq:mozero}
     \hscalar{\psi_1}{\psi_0} = \frac{\epsilon}{\pi} \int \hscalar{\psi_1}{R_0(\lambda \mp i \epsilon) R_0(\lambda \pm i \epsilon)\psi_0} \; d  \lambda.
 \end{equation}
 Hence if $\varphi_j\in\Ban$ and $\supp \chi_j \subset I_j \subset U$, then it follows from \eqref{eq:mostatic}, \eqref{eq:mozero} and from the resolvent identity \eqref{eq:resident} that 
 \begin{equation}\label{eq:Philim}
 \begin{aligned}
    \hscalar{\chi_1&(H_1) \varphi_1}{ \big(\mopm(H_1,H_0) -\idop\big) \chi_0(H_0)\varphi_0} 
   \\ &= -\lim_{\epsilon \downarrow 0}\frac{\epsilon}{\pi} \int 
     \hscalar{\varphi_1}{ \chi_1(H_1) R_1(\lambda \mp i \epsilon) V R_0(\lambda \mp i \epsilon) R_0(\lambda \pm i \epsilon) \chi_0(H_0) \varphi_0} \; d  \lambda
  \\   &= -\lim_{\epsilon \downarrow 0} \int   
      \frac{\chi_1(\mu_1) }{ \mu_1-\lambda \pm i\epsilon }  \; \frac{ \chi_0(\mu_0) \epsilon/\pi}{(\mu_0-\lambda)^2+\epsilon^2}\; \hscalar{\varphi_1}{ A_1(\mu_1) V A_0(\mu_0) \varphi_0}\,d\lambda \, d\mu_1 \, d\mu_0
   \\ &= \lim_{\epsilon \downarrow 0} \int \hscalar{\Phi_1(\lambda,\epsilon)}{\Phi_0(\lambda,\epsilon)}_\kcal\,d\lambda,
\end{aligned}
 \end{equation}
where \eqref{eq:spectralint} has been used twice, and where $\Phi_j$ are the $\kcal$-valued functions
\begin{align*}
  \Phi_1(\lambda,\epsilon) &= \int \frac{\chi_1(\mu) }{ \lambda-\mu \pm i\epsilon } V_1 A_1(\mu) \varphi_1\,d\mu,
  \\ 
  \Phi_0(\lambda,\epsilon) &= \int \frac{ \chi_0(\mu) \epsilon/\pi}{(\mu-\lambda)^2+\epsilon^2}V_0 A_0(\mu) \varphi_0\,d\mu .
\end{align*}
Note that these integrals exist as they run over the compact intervals $I_j$ where the integrand is continuous in the norm of $\kcal$. As $\epsilon\to0$, the first integral $\Phi_1(\lambda,\eps)$ has a limit by the Privalov-Korn theorem (Lemma~\ref{lemma:Privalovop}), whereas the second one evidently satisfies $\Phi_0(\lambda,\epsilon)\to \chi_0(\lambda)V_0A_0(\lambda)\varphi_0$.

Using the estimate from the Privalov-Korn theorem on $\Phi_1$ and an elementary estimate on $\Phi_0$, it moreover follows that $\|\Phi_1(\lambda,\epsilon)\|_{\kcal}\|\Phi_0(\lambda,\epsilon)\|_{\kcal}$ has an $\epsilon$-independent upper bound that is integrable in~$\lambda$. Hence we may use dominated convergence to conclude the claimed result.
\end{proof}

\subsection{High energy behaviour}

We now analyze the high-energy behaviour of resolvents and wave operators, starting with a single selfadjoint operator $H$.

\begin{definition}\label{def:orderbeta}
  Let $\beta \in (0,1)$. We say that an $\Ban$-smooth operator $H$ is of high-energy order $\beta$ if
  there exist $\hat\lambda,b >0$ such that 
  \begin{equation}\label{highR}
   \lVert  R(\lambda \pm i0) \rVert_{\Ban,\Band} \leq b |\lambda|^{-\beta} \quad \text{for all } \lambda \in U, \;|\lambda| \geq \hat\lambda.   
  \end{equation}
  We say that $H$ is of \emph{strict} high-energy order $\beta$ if, additionally, $(-\infty, - \hat \lambda]\cup [ \hat \lambda, \infty)$ is contained in $U$.
\end{definition}

\sloppy
Note that being of strict high-energy order $\beta$ implies $H \in \aac(\Hil)$. Hence, in a smooth scattering system with such operators $H_0,H_1$, mutual $f_\beta$-boundedness is the same as $f_\beta$-equivalence, since $f_\beta(H_0)-f_\beta(H_1)$ is always bounded, cf.~Lemma~\ref{lemma:fractional}. We will return to this topic shortly.

\fussy
With $R(\lambda \pm i0)$, also $A(\lambda)$ fulfill similar bounds at large $|\lambda|$, see Eq.~\eqref{eq:rdiff}. 
Vice versa, we can deduce the high-energy behaviour of $R(\lambda\pm i0)$ from that of $A(\lambda)$, given a uniform H\"older estimate.

\begin{proposition}\label{prop:ahighenergy}
 Let $H$ be a selfadjoint operator, $U\subset \rbb$  an open set of full measure and $A:U\to \bops(\Ban,\Band)$ be such that $\frac{d}{d\lambda} \hscalar{\varphi}{E(\lambda)\varphi} = \hscalar{\varphi}{A(\lambda)\varphi}$ for all $\varphi \in \Ban$, $\lambda \in U$. 
 Suppose that $ \lambda \mapsto \gnorm{  A(\lambda) }{\Ban,\Band}$ is locally integrable, that $A(\lambda)$ is locally H\"older continuous, and that there are constants $c>0$, $\beta,\theta \in(0,1)$, $\hat \lambda>0$, $q>1$ such that $(-\infty,-\hat\lambda]\cup[\hat\lambda,\infty) \subset U$ and
\begin{align*} 
 \gnorm{  A(\lambda) }{\Ban,\Band} &\leq c |\lambda|^{-\beta} 
 & \text{whenever }&
 |\lambda| \geq \hat \lambda,  
 \\ 
 \gnorm{  A(\lambda) -A(\lambda') }{\Ban,\Band} &\leq c |\lambda|^{-\beta-\theta }|\lambda-\lambda'|^\theta 
    &  \text{whenever }&
 |\lambda| \geq \hat \lambda, \; 1 < \lambda'/\lambda \leq q^2.
 \end{align*}
 Then $H$ is $\Ban$-smooth and of strict high-energy order $\beta$. 
\end{proposition} 

\begin{proof}
 $H$ is $\Ban$-smooth by Lemma~\ref{lemma:ator}. For more quantitative estimates, fix $\lambda \geq q\hat \lambda$ (the case $\lambda \leq - q\hat\lambda$ is analogous). Let $I=[\lambda/q, q\lambda]\subset U$.  For $\epsilon > 0$, we can write in the sense of weak integrals on $\Ban\times\Ban$,
 \begin{equation*}
 \begin{aligned}
    R(\lambda\pm i 0) = 
     \underbrace{\lim_{\epsilon\downarrow 0} \int\limits_{I} \frac{A(\lambda')d\lambda'}{\lambda'-\lambda-i\epsilon} }_{=:J_{1}(\lambda)} 
 + \underbrace{ \int\limits_{ \substack{|\lambda'|\geq \hat \lambda \\ \lambda' \not \in I}} \frac{A(\lambda') d\lambda'}{\lambda'-\lambda} }_{=:J_{2}(\lambda)}
 + \underbrace{ E(-\hat\lambda,\hat\lambda) (H-\lambda)^{-1}  }_{=:J_{3}(\lambda)}.
  \end{aligned}
 \end{equation*}
  To estimate these terms, we note that by our hypothesis,
 \begin{equation*}
    \sup_{\lambda' \in I} \gnorm{A(\lambda')}{\Ban,\Band} + \sup_{\lambda'\neq\lambda'' \in I} \frac{\lambda^\theta}{|\lambda'-\lambda''|^\theta}\gnorm{A(\lambda')-A(\lambda'')}{\Ban,\Band} \leq c_1 \lambda^{-\beta} 
 \end{equation*}
 with a constant $c_1>0$.
Hence the Privalov-Korn theorem (Lemma~\ref{lemma:Privalovop}) yields the estimate
 \begin{equation}\label{eq:aeest}
    \gnorm{J_{1}(\lambda)}{\Ban,\Band} \leq 
     c_1 k_\theta \lambda^{-\beta}
 \end{equation}
 with constants independent of $\lambda$.  For estimating $J_{2}$, we split the integration region further into $(-\infty,-\lambda] \cup (-\lambda, -\hat\lambda) \cup (\hat\lambda,\lambda/q) \cup [q\lambda, \infty)$. 
 We have
 \begin{equation*}
  \Big\lVert \int_{-\infty}^{-\lambda} d\lambda' \frac{A(\lambda')}{\lambda-\lambda'} \;\Big\rVert_{\Ban,\Band} \leq 
  \int_{-\infty}^{-\lambda} d\lambda' \frac{c |\lambda'|^{-\beta}}{|\lambda'|} 
  \leq c_2 \lambda^{-\beta}
 \end{equation*}
 with some $c_2>0$. A similar estimate holds for the integral over $[q\lambda, \infty)$.
 Further,
 \begin{equation*}
  \Big\lVert \int_{-\lambda}^{-\hat\lambda} d\lambda' \frac{A(\lambda')}{\lambda-\lambda'} \;\Big\rVert_{\Ban,\Band} \leq \frac{1}{\lambda}
  \int_{-\lambda}^{-\hat\lambda}  d\lambda' c |\lambda'|^{-\beta} 
  \leq c_3 \lambda^{-\beta} + c_4 \lambda^{-1}
 \end{equation*}
 with $c_3,c_4>0$. The interval $(\hat\lambda,\lambda/q)$ is handled similarly. 
Therefore, one has $\gnorm{J_{2}}{\Ban,\Band}\leq c_5 \lambda^{-\beta}$. 
Finally, it is clear that $\norm{J_{3}}_{\Ban,\Band} \leq c_6 \norm{J_{3}}_{\Hil,\Hil} \leq c_6 (\lambda-\hat\lambda)^{-1}$. 

Combined, we have shown that $R(\lambda\pm i0)$ exists for all $|\lambda|\geq q \hat\lambda$ and fulfills  $\gnorm{R(\lambda\pm i0)}{\Ban,\Band} \leq c_7 |\lambda|^{-\beta}$; therefore $H$ is of strict high-energy order $\beta$. 
\end{proof}

Now passing to smooth scattering systems, it turns out that both operators $H_j$ are always of the same high-energy order, hence we can meaningfully speak of the \emph{system} having (strict) high-energy order $\beta$.

\begin{proposition}\label{prop:boundtransfer}
  Let $(H_0,H_1,\Ban,\Hil)$ be a smooth scattering system, and let $\beta\in(0,1)$. Then $H_0$ is of (strict) high-energy order $\beta$ if and only if $H_1$ is.
\end{proposition}

\begin{proof}
  Let $U_j$ be the open set of full measure pertaining to the $\Ban$-smooth operator $H_j$, and set $V:=H_1-H_0$. Suppose that $H_0$ is of high-energy order $\beta$.  After possibly increasing $\hat\lambda$, we may assume that $\| R_0(\lambda \pm i0) \|_{\Ban,\Ban^\ast} < (2 \| V \|_{\Ban^\ast,\Ban})^{-1}$ for $\lambda \in U_0$, $|\lambda| \geq \hat\lambda$.  Hence for $z = \lambda \pm i0$ in this range and in a small neighbourhood within $\mathbb{H}_\pm$, we have $\|V R_0(z) \|_{\Ban,\Ban} \leq \frac{1}{2}$, and the Neumann series
  \begin{equation}
      (\idop +V R_0(z))^{-1} = \sum_{n=0}^\infty \big(- V R_0(z) \big)^n
  \end{equation}
  converges in $\bops(\Ban,\Ban)$. That is, $\idop+V R_0(z)$ has an inverse in $\bops(\Ban,\Ban)$, with norm at most 2, which is locally H\"older continuous in $z$ since $R_0$ is.
  Now from the resolvent equation \eqref{eq:resident}, we obtain
  \begin{equation}\label{eq:r0r1inv}
     R_1(z) = R_0(z) \big(\idop +V R_0(z)\big)^{-1}.
  \end{equation}
  Thus $U_1$ can be chosen to contain $U_0 \cap \{ \lambda : |\lambda| \geq \hat \lambda\}$, and on that set we have  
  \begin{equation*}
     \| R_1(\lambda \pm i0) \|_{\Ban,\Ban^\ast} \leq  2 \gnorm{ R_0(\lambda \pm i0) }{\Ban,\Ban^\ast} \leq 2b |\lambda|^{-\beta}
  \end{equation*}
  as claimed. The other direction follows by symmetric arguments.
\end{proof}

We will now turn our attention to the high-energy behaviour of the \moller{} operator in a smooth scattering system, and investigate $f_\beta$-boundedness and $f_\beta$-equivalence (which we write as $H_0 \sim_\beta H_1$); here $f_\beta(\lambda) =(1+\lambda^2)^{\beta/2}$ with some $\beta \in (0,1)$, as announced. 

\begin{theorem}\label{theorem:betaequiv}
  Let $(H_0,H_1, \Ban,\Hil)$ be a smooth scattering system of high-energy order $\beta\in(0,1)$. 
  Then $H_0$ and $H_1$ are mutually $f_\beta$-bounded. If the system is of \emph{strict} high-energy order $\beta$, then $H_0 \sim_\beta H_1$.
\end{theorem}

\begin{proof}
First, let $\varphi_0,\varphi_1\in\Ban$, and let $\chi_0,\chi_1:\rbb_+\to[0,1]$ be H\"older continuous such that $\supp \chi_j \subset \cup_k I_k$, where $I_k$ are finitely many disjoint compact intervals and $I_k \subset U \cap [\hat\lambda,\infty)$.
By Proposition~\ref{prop:moellersmooth}, we have
\begin{equation}\label{eq:hbetaexp}
\begin{aligned}  
\bighscalar{\chi_1(H_1) \varphi_1 }{& \big(\mopm(H_1,H_0) -\idop\big) H_0^\beta \chi_0(H_0) \varphi_0} 
\\
  &= \lim_{\epsilon\downarrow 0} \int_0^\infty d\lambda \,d\mu \frac{\chi_1(\mu)\chi_0(\lambda) \lambda^\beta}{\lambda-\mu \mp i\epsilon} \hscalar{\varphi_1}{ A_1(\mu) V A_0(\lambda) \varphi_0}
\\
&=\lim_{\epsilon\downarrow 0}\int_0^\infty d\lambda d\mu \frac{(\lambda/\mu)^{\beta/2}}{\lambda-\mu \mp i\epsilon} \hscalar{ \Phi_1(\mu)}{ \Phi_0(\lambda)}_\kcal,
  \end{aligned}
\end{equation}
  where $V=V_1\st V_0$ with $V_j: \Band \to \kcal$, and $\Phi_j$ are the $\kcal$-valued functions on $\rbb_+$ given by
\begin{equation*}
    \Phi_j(\lambda) =  \chi_j(\lambda) \lambda^{\beta/2} V_j A_j(\lambda) \varphi_j. 
\end{equation*}
   From our assumption on the high-energy order of the $H_j$, we have the estimate $\gnorm{A_j(\lambda)}{\Ban,\Band} \leq\frac{b}{\pi}\,\lambda^{-\beta}$ for all $\lambda\in U$ with $|\lambda|\geq\hat\lambda$. Hence, using \eqref{eq:albound},
   \begin{equation*}
   \begin{aligned} 
   \int_0^\infty d\lambda \gnorm{\Phi_j(\lambda)}{\kcal}^2 
   &\leq 
   \int_{\hat\lambda}^\infty d\lambda \, \chi_j(\lambda)^2 \lambda^\beta \,\gnorm{V_j}{\Band,\kcal}^2 \gnorm{A_j(\lambda)}{\Ban,\Band} \hscalar{\varphi_j}{A_j(\lambda) \varphi_j}  
   \\
   &\leq\frac{b}{\pi}\gnorm{V_j}{\Band,\kcal}^2 
   \int_{\hat\lambda}^\infty d\lambda \, \chi_j(\lambda)^2 \, \hscalar{\varphi_j}{A_j(\lambda) \varphi_j}  
   \\
   &\leq  \frac{b}{\pi} \gnorm{V_j}{\Band,\kcal }^2 \, \gnorm{\varphi_j}{\Hil}^2.  
  \end{aligned}
  \end{equation*}
  In other words, the $\Phi_j$ are elements of $L^2(\rbb_+,\kcal)$, and the constant $b>0$ in their norm is independent of our choice of $\chi_j$ under the given constraints.
  Now in \eqref{eq:hbetaexp}, $K(\lambda,\mu)=(\lambda/\mu)^{\beta/2}(\lambda-\mu \mp i0)^{-1}$ is the kernel of a bounded operator $T$ on $L^2(\rbb_+)$ by Lemma~\ref{lem:powerkernel}. Hence also $T \otimes \idop_\kcal$ is bounded on $L^2(\rbb_+ ,\kcal)$. This yields
  \begin{equation}\label{eq:foest}
  \begin{aligned}
     \lvert \hscalar{\varphi_1}{ \chi_1(H_1)  (\mo_\pm(H_1,H_0) - \idop) & H_0^\beta  \chi_0(H_0)  \varphi_0}  \rvert \\
     &\leq
      c  \norm{ \Phi_0 }_{L^2(\Rl_+,\kcal)}  \norm{ \Phi_1 }_{L^2(\Rl_+,\kcal)}   \\
      &\leq \frac{c b}{\pi} \gnorm{V_0}{\Band,\kcal} \gnorm{V_1}{\Band,\kcal} \gnorm{\varphi_0}{\Hil} \gnorm{\varphi_1}{\Hil}  
  \end{aligned}
  \end{equation}
  with a universal $c>0$ (depending only on $\beta$). 
  
  Now we can choose a sequence of $\chi_j$ of the form stated above such that $\chi_j(H_j)$ converges strongly to $\pac{j} E_j(\hat\lambda,\infty)$. Thus \eqref{eq:foest} yields,
  considering that $E_0(\hat\lambda,\infty)( H_0^\beta  - f_\beta(H_0) )$ is bounded,
  \begin{equation*}
      E_1(\hat\lambda,\infty) \pac{1} \big(\mo_\pm(H_1,H_0) - \idop \big)  f_\beta(H_0) \pac{0}  E_0(\hat\lambda,\infty) \in \boundedops.
  \end{equation*}
  Similar arguments show that the analogous expressions with one or both of the $E_j(\hat\lambda,\infty)$ swapped for $E_j(-\infty,-\hat\lambda)$ are bounded. (This requires boundedness of the integral operator with kernel $K(\lambda,\mu)=(\lambda/\mu)^{\beta/2}(\mu+\lambda)^{-1}$, see again Lemma~\ref{lem:powerkernel}.)
  Moreover, as Lemma~\ref{lemma:fractional} shows, the boundedness of $H_0-H_1$ implies that also $ f_\beta(H_0)- f_\beta(H_1)$ is bounded. Hence Lemma~\ref{lemma:hboundequiv} is applicable, and we obtain that $ \pac{1} (\mo_\pm(H_1,H_0) - \idop)  f_\beta(H_0) \pac{0} $ is bounded. The statement with $H_1$ and $H_0$ exchanged follows symmetrically; thus $H_1$ and $H_0$ are mutually $f_\beta$-bounded. In the case of strict high energy order, we have $H_j \in \aac(\Hil)$ and hence it follows that $H_1 \sim_\beta H_0$.
\end{proof}

\subsection{Pseudo-differential operators} \label{sec:doExamples}

Our results in the smooth method can be applied to a wide range of examples where $H_0$ is a differential or pseudo-differential operator and the perturbation $V=H_1-H_0$ is a multiplication operator. Here we treat $f_\beta$-equivalence for the perturbed polyharmonic operator, i.e., where $H_0$ is a fractional power of the Laplace operator, using familiar techniques for the Schr\"odinger operator ($\ell=2$ below); see, e.g., \cite{Yafaev:analytic}.

\begin{example}\label{example:laplace}
 Let $H_0=(-\Delta)^{\ell/2}$ act on its natural domain of selfadjointness in $\hcal=L^2(\rbb^n)$, where $\ell\in(1,\infty)$, $n \in \mathbb{N}$. Let $v \in L^\infty(\rbb^n)$ such that $\sup_x  (1+|x|^2)^{\alpha} \lvert v(x) \rvert  < \infty$ with some $\alpha > \frac{1}{2}$, and let $V\in\boundedops$ be the multiplication with $v$. Then $H_0$ and $H_1:=H_0+V$ are $f_\beta$-equivalent for any $0 < \beta \leq 1 - \frac{1}{\ell}$.
\end{example}
  
\begin{proof}
Let $\langle x \rangle$ be the multiplication operator by $(1+|x|^2)^{1/2}$. 
We define $\Ban\subset\Hil$ as the completion of $\mathscr{S}(\mathbb{R}^n)$ in the norm $\norm{f}_\Ban := \norm{ \langle x \rangle^{\alpha} f}_\Hil$.
To show that $H_0$ is $\Ban$-smooth, let us introduce for fixed $\lambda>0$ the map $\Gamma(\lambda):  \mathscr{S}(\mathbb{R}^n) \rightarrow L^2(S^{n-1})$ given by
\begin{equation*}
\big(  \Gamma(\lambda)f\big)(\omega) = (2\pi)^{-n/2}\int d^n x\; e^{i\langle \lambda \omega, x\rangle}f(x) = \tilde f(\lambda \omega).
\end{equation*}
By \cite[Theorem~1.1.4]{Yafaev:analytic} it extends 
to a bounded operator $\Gamma(\lambda): \Ban \rightarrow L^2(S^{n-1})$ with norm bound
\begin{equation}\label{gammanorm}
\norm{ \Gamma(\lambda)}_{\Ban, L^2(S^{n-1})} \leq C\lambda^{-\frac{n-1}{2}} 
\end{equation}
for all $\lambda >0$, where $C$ is independent of $\lambda$.
It also follows from \cite[Theorem~1.1.5]{Yafaev:analytic} that $\Gamma(\lambda)$ is locally H\"older continuous, in the sense that
\begin{equation}\label{gamma:holder}
\|   \Gamma(\lambda) - \Gamma(\lambda')    \|_{\Ban, L^2(S^{n-1})}  \leq  C' \lvert  \lambda - \lambda'  \rvert^\theta
\end{equation}
for some $\theta\in(0,1)$ (in fact $\theta=\alpha-\frac{1}{2}$ if $\frac{1}{2}<\alpha<\frac{3}{2}$), where $C'$ can be chosen uniformly for all $\lambda,\lambda'$ in a fixed compact interval in the open half line $\rbb_+$.
The derivative of the spectral measure of $H_0$ is now given by
\begin{equation*}
\frac{d}{d\lambda} \langle \varphi, E_0(\lambda) \psi \rangle = \frac{1}{\ell} \lambda^{\frac{n}{\ell}-1}\langle \varphi, \Gamma(\lambda^{1/\ell})^\ast \Gamma(\lambda^{1/\ell}) \psi\rangle =:\hscalar{\varphi}{A_0(\lambda) \psi}
\end{equation*}
for $\varphi,\psi \in \Ban$ and $\lambda >0$, and $A_0(\lambda)=0$ for $\lambda<0$.
As a consequence of \eqref{gammanorm}, $A_0(\lambda)$ is a bounded operator from $\Ban$ to $\Band$ with norm bounded by
\begin{equation}\label{eq:normA}
\norm{ A_0(\lambda)}_{\Ban, \Band} 
\leq
\hat{C} \lambda^{-1+\frac{1}{\ell}}, \quad \lambda>0,
\end{equation}
and \eqref{gamma:holder} implies that $A_0(\lambda)$ is locally H\"older continuous on $\rbb_+$ as a composition of H\"older continuous functions.
Therefore all the requirements of Lemma~\ref{lemma:ator} are satisfied with $U=\rbb\backslash\{0\}$, and we can conclude that $H_0$ is $\Ban$-smooth.

By Lemma~\ref{lemma:perturb}, also $H_1$ is then $\Ban$-smooth if we can show that $R_0(z) \in \bops(\Ban,\Band)$ is compact for $\im z \neq 0$ (note that $\Ban$ is Hilbertisable). But this is equivalent to compactness of $\langle x \rangle^{-\alpha} R_0(z) \langle x \rangle^{-\alpha}$ in $\bops(\Hil)$, which follows since this operator is a product of suitable multiplication and convolution operators \cite[Lemma~1.6.5]{Yafaev:general}.

For analyzing the high-energy behaviour of $H_0$, we define the unitary dilation operators on $\Hil$, 
\begin{equation}\label{eq:dtau}
\big( D(\tau) \varphi\big)(x) = \tau^{-n/2} \varphi(\tau^{-1}x), \quad \varphi \in \Hil,\; \tau >0.
\end{equation}
Considering them as operators from $\Ban$ to $\Ban$, or from $\Band$ to $\Band$, one finds that%
\begin{equation}\label{eq:dboundsalpha}
\norm{D(\tau)}_{\Ban,\Ban} \leq C \tau^{\alpha} \quad \text{and} \quad  \norm{D(\tau^{-1})}_{\Band ,\Band} \leq C \tau^{\alpha} \quad \text{for all} \; \tau \geq 1
\end{equation}
with some $C>0$.
One also computes that
\begin{equation}\label{eq:rscalez}
D(\tau^{-1})R_0(z)D(\tau) = \tau^{\ell} R_0(\tau^{\ell} z).
\end{equation}
Together with \eqref{eq:dboundsalpha}, we then find
\begin{equation*}
\norm{R_0(\lambda \pm i 0)}_{\Ban,\Band}  \leq C^2 \norm{R_0(1 \pm i 0)}_{\Ban,\Band}\; \lambda^{-1+\frac{2\alpha }{\ell}}
\end{equation*} 
and conclude that $H_0$ is of strict high-energy order $\beta = 1-\frac{2\alpha}{\ell}$, provided this number is positive. By Theorem~\ref{theorem:betaequiv}, we then have $H_0\sim_\beta H_1$. Since $\alpha>\frac{1}{2}$ was arbitrary and the $\Ban$-norm becomes stronger with increasing $\alpha$, we have thus shown our claim for all $0 < \beta < 1-\frac{1}{\ell}$.

(Alternatively, we may deduce this as follows: Pick an interval $[1,q^2]$ where  $A_0(\lambda)$ is uniformly H\"older continuous. Using $D(\tau^{-1})A_0(\lambda) D(\tau) = \tau^\ell A_0(\tau^\ell \lambda)$ and \eqref{eq:dboundsalpha}, a scaling argument like above shows that
the hypothesis of Proposition~\ref{prop:ahighenergy} is satisfied, yielding the result.)

\fussy
Proving the claim for $\beta = 1-\frac{1}{\ell}$ requires more effort; we sketch the argument. We make use of the Agmon-H\"ormander space $\Ah\subset\Hil$; see \cite[Secs.~6.3 and 7.1]{Yafaev:analytic} for its definition and properties. Here we need only that $\Ban\subset\Ah\subset\Hil$ are continuous dense inclusions, and that, in some improvement over \eqref{eq:dboundsalpha},
\begin{equation}
\norm{D(\tau)}_{\Ah,\Ah} \leq C \tau^{1/2} \quad \text{and} \quad  \norm{D(\tau^{-1})}_{\Ahd,\Ahd} \leq C \tau^{1/2} \quad \text{for all} \; \tau \geq 1.
\end{equation}
We will show below that $R_0(\lambda \pm i0)$ is bounded from $\Ah$ to $\Ahd$ for each fixed $\lambda >0$. A scaling argument as above then yields
\begin{equation*}
\norm{R_0(\lambda \pm i 0)}_{\Ah,\Ahd}  \leq C^2  \norm{R_0(1 \pm i 0)}_{\Ah,\Ahd} \; \lambda^{-1+\frac{1}{\ell}}
\end{equation*}
for all $\lambda \geq 1$. An analogous estimate holds for $\norm{R_0(\lambda \pm i 0)}_{\Ban,\Band}$, since the inclusion $\Ban\subset\Ah$ is continuous. Hence $H_0$ is of strict high-energy order $\beta=1-\frac{1}{\ell}$, and $H_0\sim_\beta H_1$.

In order to show that  $R_0(\lambda \pm i0)\in \bops(\Ah,\Ahd)$, let us define for $\epsilon > 0$,
\begin{equation}\label{LR}
S_{\pm\epsilon} := L R_0^{(2)}(\lambda^{2/\ell} \pm i \epsilon) \in \bops(\Hil),
\end{equation}
where $R_0^{(2)}( {\cdot} )$ is the resolvent of $-\Delta$, and $L$ is the multiplication operator in Fourier space by the function
\begin{equation*}
\hat\ell(\xi) := \frac{\lvert \xi \rvert^2 - \lambda^{2/\ell}}{\lvert \xi\rvert^\ell - \lambda}.
\end{equation*}
One notices that $\hscalar{\varphi }{S_{\pm\epsilon} \psi} \to \hscalar{\varphi}{R_0(\lambda\pm i0)  \psi}$ for each $\varphi,\psi\in\Ss(\rbb^n)\subset \Ah$ as $\epsilon \to 0$. 
On the other hand, \cite[Theorem~6.3.3]{Yafaev:analytic} yields that $\norm{S_{\pm\eps}}_{\Ah,\Ahd}$ is uniformly bounded for all $\epsilon \leq 1$.
Hence, $\lvert \hscalar{\varphi}{R_0(\lambda\pm i0)  \psi}\rvert \leq c \norm{\varphi}_{\Ah}\norm{\psi}_{\Ah}$ for all $\varphi,\psi\in \Ss(\rbb^n)$, and $R_0(\lambda\pm i0)$ extends to a bounded operator from $\Ah$ to $\Ahd$.
\end{proof}

\begin{remark}
Theorem~6.3.3 in \cite{Yafaev:analytic} assumes that the function $\hat\ell$ is smooth everywhere, but this is not essential for its proof; it suffices that, as in our case, the function is smooth outside the origin $\xi=0$, and bounded in a neighbourhood of the origin. 
\end{remark}
\stepcounter{definition}

The estimates on the high-energy order, $0<\beta \leq 1 - \frac{1}{\ell}$, cannot be improved in general, as the following special case shows.

\begin{example}\label{example:s1d}
 In Example~\ref{example:laplace}, let $n=1$, $\ell=2$, and suppose that $v\neq 0$ is compactly supported and nonnegative. Then $(H_1,H_0)$ is \emph{not} $f_\beta$-bounded for any $ \beta > \frac{1}{2}$.  In particular, $H_1 \not\sim_\beta H_0$.
\end{example}
 
\begin{proof}
First, note that $\pac{0}=\idop$; also, since $v\geq 0$, we know that $H_1 \geq 0$ and hence $H_1$ does not have eigenvalues \cite[Lemma~6.2.1]{Yafaev:analytic}, i.e., $\pac{1}=\idop$.
 
Now let $v$ be supported in the compact interval $[a,b]$. We choose $\varphi\in C_0^\infty(b,\infty)$ and $\psi\in L^2(\rbb)$ with its Fourier transform $\tilde\psi \in C_0^\infty(\rbb_+)$ such that $\hscalar{\varphi}{\psi} \neq 0$. For $n \in \nbb$, define $ \varphi_n (x):=e^{inx}\varphi(x)$,  $\psi_n(x):=e^{inx}\psi(x)$. The \moller{} operator $\mo:=\mopm(H_1,H_0)$ can in our case  be written as 
\begin{equation*}
(\mo \psi)(x) = \frac{1}{\sqrt{2\pi}} \int dk\, m(x,k) T(k) e^{ikx} \tilde\psi(k)
\end{equation*}
with a complex-valued function $T$ and a certain integral kernel $m$, which for $x>b$ is given by $m(x,k)=1$  \cite[Sec.~2]{DT:scattering}. Hence we have
\begin{equation}
\begin{aligned}\label{eq:omega1d}
\langle  \varphi_n, (\mo - \idop) &  f_\beta(H_0) \psi_n \rangle \\
&= \int dx\, \overline{\varphi_n(x)} \int dk\, \big(m(x,k) T(k) -1\big)e^{ikx}  (1+k^4)^{\beta/2} \tilde\psi_n(k)  \\
&=
\int dk\, \overline{\tilde\varphi(k)} \big(T(k +n) -1\big) (1+(k + n)^4)^{\beta/2}\tilde\psi(k).
\end{aligned}
\end{equation}
Now by \cite[Proof of Thm.~1.IV]{DT:scattering}, $T$ has the asymptotics
\begin{equation*}
  T(k) = 1+ \frac{\int dx \, v(x)}{2ik} + O(k^{-2}) \quad \text{as } k \to \infty,
\end{equation*}
where $\int dx \, v(x) \neq 0$ by hypothesis.
Since $\beta>1/2$, we find that \eqref{eq:omega1d} diverges as $n \to \infty$, while $\norm{\varphi_n}$ and $\norm{\psi_n}$ are independent of $n$.
\end{proof}

\subsection{Tensor products}\label{sec:tensor}

We now ask whether the high-energy order of an operator is stable under taking tensor products, in the following sense: Let $H_A$ be a selfadjoint operator on $\Hil_A $ which is smooth with respect to some Gelfand triple $\Ban_A \subset \Hil_A \subset \Band_A$. Let $H_B$ be another selfadjoint operator on a Hilbert space $\Hil_B$, assumed to have purely discrete spectrum. On $\Hil:=\Hil_A \otimes \Hil_B$, consider $H := H_A \otimes \idop + \idop \otimes H_B$.

In the following, we will always denote the resolvent of $H_A$ as $R_A(z)$, etc. We note that, if $H_B=\sum_j \lambda_j P_j$ is the spectral decomposition of $H_B$, then
\begin{equation}\label{eq:rtensor}
   R(z) = \sum_j R_A(z- \lambda_j) \otimes P_j.
\end{equation}
at least weakly on $\Hil \times \Hil$; cf.~\cite[Sec.~5.1]{BenArtzi:smooth}. The same relation then holds with $z$ replaced with $\lambda \pm i 0$ as long as $\lambda -\lambda_j \in U_A$ for all $j$, and at least in the sense of matrix elements between vectors of the form $\psi_A \otimes \psi_B$ where $\psi_A\in\Ban_A$ and $\psi_B$ is an eigenvector of $H_B$. (We will clarify below when the limit exists in the norm sense.)

For simplicity, we first treat the case of a finite-dimensional space $\Hil_B$.

\begin{proposition}\label{prop:tensorfinite}
Let $H_A$ be a selfadjoint operator on a separable Hilbert space $\mathcal{H}_A$ which is $\Ban_A$-smooth and of strict high-energy order $\beta\in(0,1)$.
Let $H_B$ be a selfadjoint operator on a finite-dimensional Hilbert space $\mathcal{H}_B$.

Set $\mathcal{H} := \mathcal{H}_A \otimes \mathcal{H}_B$ and $\Ban:=\Ban_A \otimes \Hil_B \subset\Hil$. 
Then $H := H_A \otimes \idop + \idop \otimes H_B$ is $\Ban$-smooth and of strict high-energy order $\beta$.
\end{proposition}

\begin{proof}
Set $N:= \mathbb{R} \backslash U_A$ and $S:= \sigma(H_B) + N$ (both closed null sets), and let $U:= \mathbb{R} \backslash S$. We will show that $R$ is locally H\"older continuous on $U \pm i [0,\infty)$; clearly it suffices to show this in a neighbourhood of each real point. 

Since $U$ is open, we can for any given $\lambda \in U$ find a neighbourhood $V_\lambda = [\lambda - \delta, \lambda + \delta] \pm i [0, \epsilon]$ ($\delta,\epsilon>0$) such that $V_\lambda-\lambda_j \subset U_A\pm i[0,\infty)$ for all $j$. 
We now employ \eqref{eq:rtensor} and use local H\"older continuity of $R_A$ to estimate 
\begin{equation}\label{eq:rdiff-fin}
\begin{aligned}
\| R(z') - R(z'') \|_{\Ban, \Band} 
&\leq  \sum_j  \|  R_{A}(z' - \lambda_j) - R_A(z'' - \lambda_j ) \|_{\Ban_A, \Band_A}
\\
&\leq \sum_j c_j |z'-z''|^{\theta_j}
\end{aligned}
\end{equation}
for all $z',z''\in V_\lambda$, with constants $c_j>0$, $\theta_j \in (0,1)$. Since the sum is finite, this proves local H\"older continuity of $R$. 
In particular, the limits $R(\lambda \pm i0)$ exist in $\bops(\Ban,\Band)$ for all $\lambda \in U$. Hence $H$ is $\Ban$-smooth.

For the high-energy behaviour of $R$, we note that $R(\lambda \pm i 0)$ exists for sufficiently large $|\lambda|$, and we can estimate using the high-energy order of $H_A$,
\begin{equation}\label{eq:rhoelder-fin}
\begin{aligned}
\| R(\lambda \pm i 0) \|_{\Ban, \Band} 
&\leq  \sum_j  \|  R_{A}(\lambda - \lambda_j \pm i 0) \|_{\Ban_A, \Band_A}
\\
&\leq \sum_j c_j |\lambda - \lambda_j|^{-\beta} \leq c' |\lambda|^{-\beta}
\end{aligned}
\end{equation}
which shows that $H$ is also of strict high-energy order $\beta$.
\end{proof}

We now aim at a similar result for infinite-dimensional spaces $\Hil_B$, which requires more care. We will need stronger uniformity assumptions on the bounds on $R_A$, as well as some restrictions on the spectrum of $H_B$. In order to avoid technical complications with the tensor product, we also assume that $\Ban_A$ is a Hilbert space. The result will be weaker inasmuch as the high-energy order of the sum operator is no longer known to be strict.

\begin{theorem}\label{thm:tensorinfinite}
Let $\Ban_A\subset\Hil_A\subset \Band_A$ a Gelfand triple with a Hilbert space $\Ban_A$, and $H_A$ a selfadjoint operator on $\mathcal{H}_A$. Let $H_B$ be another selfadjoint operator on a Hilbert space $\mathcal{H}_B$. Suppose that:

\begin{enumerate}[(a)]
%
\item  \label{it:unihoelder} 
$H_A$ is $\Ban_A$-smooth; moreover, there exist $\Lambda >0$, $\theta \in ( 0, 1 )$, $c >0$ and $\epsilon >0$ such that 
\begin{equation*}
\| R_A(z) - R_A(z')  \|_{\Ban_A, \Band_A} \leq c \lvert z- z'   \rvert^\theta
\end{equation*}
whenever $\lvert \operatorname{Re} z^{(\prime )} \rvert \geq \Lambda$,  $\pm \operatorname{Im} z^{(\prime )} \in [0,\epsilon]$, and $|z-z'|\leq 1$.

\item \label{it:unihe}

$H_A$ is of strict high-energy order $\beta \in (0,1)$; 
moreover, there exists $c>0$ such that for all $\lambda \in U_A$,
\begin{equation*}
\| R_A(\lambda\pm i0) \|_{\Ban_A, \Band_A} \leq c( 1 + \lambda^2)^{-\beta /2} .
\end{equation*}

\item\label{it:trace}
 There exists $\gamma>0$ such that $(1 + H_B^2)^{-\gamma + \beta/2}$ is of trace class.
\end{enumerate}

\noindent
Set $\mathcal{H} := \mathcal{H}_A \otimes \mathcal{H}_B$ and let $\Ban \subset \mathcal{H}$ be the Hilbert space with the following norm:
\begin{equation}\label{eq:xnormgamma}
\|  \cdot \|_\Ban = \|  \cdot \|_{\Ban_A} \otimes \| (1+ H_B^2)^{\gamma/2} \cdot  \|_{\mathcal{H}_B}. 
\end{equation}
Then $H := H_A \otimes \idop + \idop \otimes H_B$ is $\Ban$-smooth and of high-energy order $\beta$.
\end{theorem}

\begin{proof}
First note that $|\lambda_j| \to \infty$ due to (\ref{it:trace}), hence $\sigma(H_B)$ is locally finite; and $\Rl\backslash U_A$ is closed and bounded by (\ref{it:unihe}). Hence $U$ as defined in the proof of Prop.~\ref{prop:tensorfinite} is still open, and for given $\lambda\in U$ we can choose a compact complex neighbourhood $V_\lambda$ such that $V_\lambda-\lambda_j \subset U_A\pm i [0,\infty)$ for all $j$.
Analogous to \eqref{eq:rdiff-fin}, but now with the modified norm \eqref{eq:xnormgamma}, we obtain the estimate for $z',z''\in V_\lambda$,
\begin{equation}\label{eq:rdiff-inf}
\| R(z') - R(z'') \|_{\Ban, \Band} 
\leq  \sum_j  (1+\lambda_j^2)^{-\gamma} \|  R_{A}(z' - \lambda_j) - R_A(z'' - \lambda_j ) \|_{\Ban_A, \Band_A}.
\end{equation}
We split the sum into those $j$ where $|\lambda-\lambda_j| \leq \Lambda+1$ (with $\Lambda$ as in condition (\ref{it:unihoelder})) and their complement. Since $|\lambda_j|\to\infty$, the first mentioned sum is finite and can be estimated as in \eqref{eq:rdiff-fin}. For the remaining sum, we use (\ref{it:unihoelder}) to show (if $V_\lambda$ was chosen sufficiently small so that $|z'-z''|\leq 1$),
\begin{equation}
\begin{aligned}\label{sums}
 \sum_{j :\; \lvert \lambda - \lambda_j \rvert \geq \Lambda+1} (1 + \lambda_j^2)^{-\gamma} &  \| R_A(z' - \lambda_j) - R_A(z'' - \lambda_j)   \|_{\Ban_A, \Band_A}
 \\&\leq \sum_j (1 + \lambda_j^2)^{-\gamma} c \lvert z'- z''   \rvert^\theta \leq c' \lvert z'- z''   \rvert^\theta
\end{aligned}
\end{equation}
with a finite $c'>0$, since $(1+H_B^2)^{-\gamma}$ is trace class. In conclusion, $R(z)$ is locally H\"older continuous (also at the boundary $U\pm i0$), hence $H$ is $\Ban$-smooth.

For the high-energy order, we estimate for $\lambda\in U$ using condition (\ref{it:unihe}),
\begin{equation*}
\begin{aligned}
 \| R(\lambda \pm i0) \|_{\Ban, \Band} 
&\leq  \sum_j  (1+\lambda_j^2)^{-\gamma} \|  R_{A}(\lambda  - \lambda_j \pm i0) ) \|_{\Ban_A, \Band_A}
\\
&\leq c \sum_j (1+(\lambda-\lambda_j)^2)^{-\beta/2}  (1+\lambda_j^2)^{-\gamma}
\\
&\leq c' (1+\lambda^2)^{-\beta/2}  \sum_j (1+\lambda_j^2)^{-\gamma+\beta/2}.
\end{aligned}
\end{equation*}
(We have used the inequality $\frac{1}{1+(x-y)^2} \leq 2 \frac{1+x^2}{1+y^2}$ for $x,y\in\rbb$.) The series here is convergent due to (\ref{it:trace}). Thus $H$ is of high-energy order $\beta$.
\end{proof}

As usual, the detailed estimates in the previous theorem can be explicitly verified only in very simple examples. However, a perturbation argument as in Lemma~\ref{lemma:perturb} allows us to extend them:
\begin{corollary}\label{corr:tensorperturb}
   In the situation of Proposition~\ref{prop:tensorfinite} or Theorem~\ref{thm:tensorinfinite}, suppose that $R_A(z) \in \mathrm{FA}(\Ban,\Band)$ for every $z \in \cbb\backslash \rbb$. If $V=V^\ast\in\Gamma_2(\Band,\Ban)$, then the tuple $(H,H+V,\Ban,\Hil)$ is a smooth scattering system of high-energy order $\beta$. In particular, $H$ and $H+V$ are mutually $f_\beta$-bounded. In the situation of Proposition~\ref{prop:tensorfinite}, the system is of strict high-energy order $\beta$ and $H \sim_\beta H+V$.
\end{corollary}

\begin{proof}
   Note that the spectral projectors $P_j$ of $H_B$ have finite rank; in the case of Theorem~\ref{thm:tensorinfinite}, this follows from condition (\ref{it:trace}). 
   Thus every term $R_A(z)\otimes P_j$ in the series \eqref{eq:rtensor} lies in $\mathrm{FA}(\Ban,\Band)$. On the other hand, the series converges absolutely in $\bops(\Ban,\Band)$: In the first situation, this is trivial; in the second situation, note that $\norm{ R_A(z) }_{\Ban_A,\Band_A}\leq \norm{ R_A(z) }_{\Hil_A,\Hil_A} \leq \lvert \im z \rvert^{-1}$, and hence
   \begin{equation*}
       \sum_j \| R_A(z-\lambda_j) \otimes P_j\|_{\Ban,\Band} \leq \lvert\im z \rvert^{-1} \sum_j (1+\lambda_j^2)^{-\gamma} < \infty.
   \end{equation*}
   Thus $R(z) \in \mathrm{FA}(\Ban,\Band)$ since this space is norm-closed. The statement now follows from Lemma~\ref{lemma:perturb} and Theorem~\ref{theorem:betaequiv}.
\end{proof}

This allows for immediate applications in the finite-dimensional case for $\Hil_B$, for example, if $H_A = -\Delta$ (see Example~\ref{example:laplace}) and $H_B$ is some Hermitean matrix. In the context of quantum physics, $H$ would then be the free Schr\"odinger operator for a particle with inner degrees of freedom, and $V$ a matrix-valued scattering potential.

We will now give some more concrete examples in the infinite-dimensional case, where the conditions on $H_A$ are more delicate. 

\begin{example}\label{example:dilation}
Let $H_ A= - \Delta$ act on its natural domain of selfadjointness in $\hcal_A=L^2(\rbb^3)$ and $H_B = -\Delta$ act on $\hcal_B = L^2(S^2)$, where $S^2$ is the two-dimensional sphere. Let $v : \rbb \to \bops(\hcal_B)$, $v(x)=v(x)^\ast$, such that 
\begin{equation}\label{eq:tensorpot}
\sup_x  \; (1+|x|^2)^{\alpha} \;  \big\lVert (1+ H_B^2)^{-\gamma/2} v(x)  (1+ H_B^2)^{-\gamma/2} \big\rVert_{ \hcal_B,\hcal_B } < \infty 
\end{equation}
with some $\alpha > 1$, $\gamma > ( \beta+1) / 2 $ and $\beta \leq 1/2$, and let $V\in \bops (\hcal_A \otimes \hcal_B)$ be the operator multiplying with $v$.
Then, with  $H:=H_A \otimes  \idop +  \idop \otimes H_B$, we have that $H$ and $H + V$ are mutually $f_\beta$-bounded.
\end{example}

\begin{proof}
We aim to apply Theorem~\ref{thm:tensorinfinite} and Corollary~\ref{corr:tensorperturb}; let $\Ban$ be as defined there, with $\|f\|_{\Ban_A}=\|\langle x\rangle^\alpha f \|_{\Hil_A}$. 

For condition (\ref{it:unihoelder}) in the theorem, note that $H_A$ is $\Ban_A$-smooth by Example~\ref{example:laplace}. For the more detailed estimate, it suffices to consider $\alpha<\frac{3}{2}$ and $\im z^{(\prime)} \geq 0$.
We use the dilation operators $D(\tau)$ defined in \eqref{eq:dtau} and the relation \eqref{eq:rscalez} to show that for every $\tau\geq 1$,
\begin{equation*}
\begin{aligned}
\| &R_A (z) - R_A(z')  \|_{\Ban_A, \Band_A} \\
& \leq \tau^{-2}\| D(\tau^{-1})\|_{\Band_A,\Band_A}  \|R_A(\tau^{-2}z) - R_A(\tau^{-2}z')\|_{\Ban_A, \Band_A}   \| D(\tau) \|_{\Ban_A, \Ban_A} \\
& \leq  C \tau^{2\alpha-2}   \|R_A(\tau^{-2}z) - R_A(\tau^{-2}z')\|_{\Ban_A,\Band_A}.
\end{aligned} 
\end{equation*}
Since however $R_A(z)$ is locally H\"older continuous with exponent  $\theta=\alpha-\frac{1}{2}$, it fulfills a uniform H\"older estimate on, say, the compact region $[1,2] \pm i [0, 1]$. Choosing $\tau = (\re z)^{1/2}$, we thus have
\begin{equation*}
\| R_A (z) - R_A(z')  \|_{\Ban, \Band} 
\leq  C (\operatorname{Re}z)^{\alpha-1 -\theta} \lvert z- z'  \rvert^\theta
\leq  C \lvert z- z'  \rvert^\theta
\end{equation*}
whenever $1 \leq \re z \leq \re z' \leq 2 \re z$ and $0 \leq \im z^{(\prime)} \leq 1$; likewise for $z$ and $z'$ exchanged. This includes the region $\re z^{(\prime)} \geq 1$, $0 \leq \im z^{(\prime)} \leq 1$, $|z-z'|\leq 1$, hence (\ref{it:unihoelder}) follows. 

To show condition (\ref{it:unihe}), note that $ \| R_A(\lambda \pm i0)\|_{\Ban_A, \Band_A} \leq C \lvert  \lambda \rvert^{-1/2}$ for $|\lambda|\geq 1$ by Example~\ref{example:laplace}. For $|\lambda| \leq 1$, we use the fact that $\| R_A(\lambda \pm i0)\|_{\Ban_A, \Band_A}$ is globally bounded in $\lambda$ \cite[Proposition~7.1.16]{Yafaev:analytic}; it enters here that $\alpha>1$ and that we consider the Laplacian on $\rbb^3$. 

Regarding condition (\ref{it:trace}): Since $\sigma(H_B) = \{\ell (\ell +1) \}_{\ell \in \mathbb{N}_0}$ with degeneracy $2\ell + 1$, we can compute
\begin{equation*}
\begin{aligned}
  \operatorname{tr} (1+H_B^2)^{-\gamma+\beta/2} &= \sum_{\ell \in \mathbb{N}_0} (2 \ell +1) (1 + \ell^2 (\ell +1)^2)^{-\gamma + \beta/2} 
  \\ &\leq 2 \sum_{\ell \in \mathbb{N}_0} (1+\ell)^{-4\gamma +2 \beta +1},
\end{aligned}
\end{equation*}
which converges for $4\gamma>2\beta+2$ as in the hypothesis.%
---
In conclusion, Theorem~\ref{thm:tensorinfinite} applies. 
Also, we already noted in Example~\ref{example:laplace} that $R_A(z)$ is compact in $\bops(\Ban,\Band)$ for $z \in \cbb\backslash \rbb$, and we have $V\in \Gamma_2 (\Band,\Ban)$ by assumption \eqref{eq:tensorpot}. Hence we can apply Corollary~\ref{corr:tensorperturb} and conclude that $H$ and $H+V$ are mutually $f_\beta$-bounded.
\end{proof}

Similar methods would apply to Laplace operators in higher dimensions ($n \geq 3$), but not for $n < 3$, since in that case there is no uniform bound on $\gnorm{R_A(z)}{\Ban_A,\Ban_A^\ast}$ near $z=0$. 

Let us focus on the one-dimensional Laplacian here. Instead of the ``free'' operator $-\Delta$, one can consider $H_A = - \Delta + V_A$ where $V_A$ is multiplication with a nonnegative, sufficiently rapidly decaying function; its resolvent behaves better near $z=0$, so that we can obtain a result similar to the 3-dimensional case.

\begin{example}\label{example:tensors1}
\sloppy
Let $H_ A= - \Delta + V_A$ acting on $\hcal_A=L^2(\rbb)$, where $V_A$ is multiplication with a nonnegative, compactly supported function $v_A$ in $L^\infty(\Rl) \backslash \{0\}$. Let $\beta \in (0, \frac{1}{2}]$, and let $H_B$ be another selfadjoint operator on a separable Hilbert space $\mathcal{H}_B$  such that $(1 + H_B^2)^{-\gamma + \beta/2}$ is of trace class for some $\gamma >0$.
Let $v : \rbb \to \bops(\hcal_B)$, $v(x)=v(x)^\ast$, such that 
\begin{equation*}
\sup_x  \; (1+x^2)^{\alpha} \;\big\lVert (1+ H_B^2)^{-\gamma/2} v(x)   (1+ H_B^2)^{-\gamma/2} \big\rVert_{ \hcal_B, \hcal_B  } < \infty 
\end{equation*}
with some $\alpha > \frac{3}{2}$; and let $V \in\bops (\hcal_A \otimes \hcal_B)$ be the operator multiplying with $v$. Then $H:=H_A \otimes  \idop +  \idop \otimes H_B$ and $H+V$ are mutually $f_\beta$-bounded.
\end{example}

\begin{proof}

First, as the hypothesis becomes only stronger with increasing $\alpha$, we can assume without loss of generality that $\frac{3}{2}<\alpha<2$.

Now let $\Ban_A\subset \Hil_A$ be once more defined by its norm $ \|  \cdot \|_{\Ban_A} = \|  \langle x\rangle^\alpha \cdot \|_{L^2(\rbb)}$ and let $\mathcal{H}$ and $\Ban \subset \mathcal{H}$ be as in Eq.~\eqref{eq:xnormgamma}.
As before, we aim to verify conditions (a)--(c) of Theorem~\ref{thm:tensorinfinite} in our situation.

For (\ref{it:unihoelder}), we know from Example~\ref{example:laplace} that $H_A = -\Delta+V_A$ is $\Ban_A$-smooth. In fact, since $H_A\geq 0$ cannot have negative eigenvalues, we can choose $U_A = \mathbb{R} \backslash \{ 0\}$ \cite[Lemma~6.2.1]{Yafaev:analytic}. Now let $R_0(z)$ be the resolvent of the negative Laplacian on $\Hil_A$.
Like in \eqref{eq:r0r1inv}, we write  $R_A (z) = R_0(z) F(z)$ with $F(z) := (\idop - G(z))^{-1}$ and $G(z) :=V_A R_0 (z)$, for any $z$ where the inverse exists (we will clarify this below). As in Example~\ref{example:dilation}, $R_0$ fulfills a uniform H\"older estimate
\begin{equation}\label{eq:r0hoeld}
\begin{aligned}
   \gnorm{ R_0(z) - R_0(z')  }{\Ban,\Band}
   & \leq C \lvert z \rvert^{\alpha-1 -\theta}\lvert z - z'  \rvert^\theta
   \\ & \text{whenever  $1 \leq  \operatorname{Re}z^{(\prime)}$, $|z-z'| \leq 1$.}  
\end{aligned}
\end{equation}
Since $\alpha > \frac{3}{2}$, we can choose any $\theta <1$ here (cf.~\cite[Proposition~1.7.1]{Yafaev:analytic}). 
An analogous H\"older estimate then holds for $G(z)$ in $\gnorm{ \cdotarg }{\Ban,\Ban}$.

Further, since $\gnorm{R_0(z)}{\Ban,\Band}$ decays at large $|z|$, we can choose $\Lambda_0 >0$ such that $\| G(z) \|_{\Ban,\Ban} \leq \frac{1}{4}$ for all $\operatorname{Re} z \geq \Lambda_0$; the inverse $F(z)=(1-G(z))^{-1}$ then exists as a convergent Neumann series, and $\|F(z) \|_{\Ban,\Ban} \leq \frac{4}{3}$. To obtain H\"older estimates for $F$, we note the identity
\begin{equation*}
\begin{aligned}
F(z) - F(z') 
&= F(z')\Big( [\idop -(G(z) - G(z'))F(z')]^{-1} - \idop \Big) \\
&= F(z') \sum_{k=1}^\infty \Big( (G(z) - G(z'))F(z') \Big)^k 
\end{aligned}
\end{equation*}
where the series converges by the above estimates. 
By taking norms we obtain:
\begin{equation*}
\begin{aligned}
\| F(z) - F(z')  \|_{\Ban, \Ban}  \leq   C' \gnorm{ G(z) &- G(z') }{\Ban,\Ban}
\leq C'' \lvert z \rvert^{\alpha-2}\lvert z - z'  \rvert
 \\  &\text{whenever  $\Lambda_0 \leq  \operatorname{Re}z,  \operatorname{Re}z'$, $|z-z'| \leq 1$,}  
   \end{aligned}
\end{equation*}
where the factor $ \lvert z \rvert^{\alpha-1 -\theta}$ is decreasing (for suitable $\theta$, noting $\alpha<2$).
Hence, as $R_0$ fulfills \eqref{eq:r0hoeld} and is bounded in $\gnorm{\cdotarg}{\Ban,\Band}$ in the relevant region, we know that $R_A(z)=R_0(z)F(z)$ fulfills an analogous H\"older estimate, which proves (\ref{it:unihoelder}).

\sloppy
Regarding condition (\ref{it:unihe}), we know from Example~\ref{example:laplace} that the resolvents $R_A$ fulfill $\| R_A(\lambda \pm i0)\|_{\Ban_A, \Band_A} \leq C  \lvert \lambda \rvert^{-1/2}$ for large $|\lambda|$, and from part (a) above that $R_A$ is continuous in this norm on $\mathbb{R}\backslash\{0\}$. Hence it only remains to show that $\| R_A(z)\|_{\Ban_A, \Band_A} $ is bounded in a neighbourhood of $z=0$. To that end, recall that the integral kernel of $R_A(z)$ is given by \cite[Ch.~5]{Yafaev:analytic} as
\begin{equation}\label{resolvent1d}
   R_A(x,x';z) = R_A(x',x;z) =  \frac{ \theta_1(x,\sqrt{z})\theta_2(x',\sqrt{z}) }{\omega(\sqrt{z})}\quad \text{for $x>x'$},
\end{equation}  
where the functions $\theta_{1,2}(x,\zeta)$ with $\im\zeta\geq 0$ are the solutions of the differential equation $(-\partial_x^2 + v_A(x) - \zeta^2 ) \theta_j (x,\zeta) = 0$  with asymptotics $\theta_j(x,\zeta)=e^{\pm ix\zeta}+o(1)$, $\partial_x\theta_j(x,\zeta) = \pm i\zeta e^{\pm ix\zeta}(1+o(1))$ for $x \to \pm \infty$ (here $+$ for $j=1$ and $-$ for $j=2$), and where $\omega$ is the Wronskian of $\theta_1,\theta_2$, a continuous function of $\zeta$, also at $\zeta=0$. Note that the solutions $\theta_j(x,0)$ are real. In our case, since $v_A(x)\geq 0$, the $\theta_j(x,0)$ must be convex, and not constant as $v_A$ does not vanish identically; hence for $x$ to the right of the support of $v_A$, one has $\theta_1(x,0)=1$ and $\theta_2(x,0)=cx+d$ with some $c \neq 0$, and the Wronskian $\omega(0)$ does not vanish. Therefore we can choose a neighbourhood $U$ of 0 such that $|\omega(\sqrt{z})|\geq \epsilon>0$ there.

\fussy
Now by the Cauchy-Schwarz inequality, 
\begin{equation*}
\begin{aligned}
  \| R_A(z)\|_{\Ban_A, \Band_A} &=  \|  \langle x \rangle^{-\alpha} R_A(z) \langle x \rangle^{-\alpha} \|_{\Hil_A,\Hil_A}
  \\
  & \leq \big\lvert \omega(\sqrt{z}) \big\rvert^{-1} \prod_{j=1,2} \Big(\int dx\, (1+x^2)^{-\alpha} \big \lvert \theta_j(x,\sqrt{z}) \big\rvert^2 \Big)^{1/2}.   
\end{aligned}
\end{equation*}
Using \cite[Lemma~2.1(ii)]{DT:scattering}, we can estimate $|\theta_j(x,\zeta)|^2 \leq c (1+x^2)$ for all $\zeta$. Hence for $z\in U$,
\begin{equation*}
  \| R_A(z)\|_{\Ban_A, \Band_A} 
  \leq \frac{c}{\epsilon}  \int dx\, (1+x^2)^{-\alpha+1}.   
\end{equation*}
This integral converges since $\alpha>3/2$. Hence condition (\ref{it:unihe}) holds.

Property (\ref{it:trace}) follows directly from the hypothesis on $H_B$, hence Theorem~\ref{thm:tensorinfinite} is applicable. 
Further, $R_A(z) = R_0(z) F(z)$ is compact for $\im z \neq 0$, and $V$ is bounded in the norm of $\bops(\Band,\Ban)$ by assumption, hence mutual $f_\beta$-boundedness of $H$ and $H+V$ follows from Corollary~\ref{corr:tensorperturb}.
\end{proof}

\section{\sectioncase{Application to semibounded operators and quantum inequalities}}\label{sec:qi}

In this concluding section, we highlight an application of our results that is of interest in the context of quantum physics.
We deal with the following question. 

\sloppy
Suppose that $A$ is a selfadjoint, unbounded operator on a (separable) Hilbert space $\Hil$, and $B \in \boundedops$, such that the compression $B^\ast A B$ is (semi)-bounded. Does this (semi-)boundedness transfer to the compression $B \mopm^\ast A \mopm B $ or -- closely related -- to $(\mopm^\ast B \mopm)^\ast A (\mopm^\ast B \mopm)$, where $\mopm$ is the \moller{} operator of some scattering situation? 

\fussy
We can give a sufficient criterion for this in our context. 

\begin{theorem}\label{thm:scatterineq}
 Let $H_0,H_1 \in \aac(\Hil)$, $f \in C(\Rl)$ such that $H_0 \sim_f H_1$. Denote $\mopm=\mopm(H_1,H_0)$. 
 Let $B$ be a bounded operator and $A$ be a selfadjoint (unbounded) operator such that $\norm{  A f(H_j)^{-1} } < \infty$ for $j=0,1$. 
 Then $B^\ast A B$ is bounded above (below) iff  $B^\ast \mo_\pm^\ast A \mo_\pm B$ is bounded above (below). 
\end{theorem}

\begin{proof} 
 It suffices to show that $\mo_\pm^\ast A \mo_\pm - A$ is bounded. But as $\| \mopm\| \leq 1$, we have
  \begin{equation*}
  \begin{aligned}
    \| \mo_\pm^\ast A \mo_\pm -  A \| 
    &= \lVert (\mopm-\idop)^\ast A \mopm + A (\mopm-\idop) \rVert 
    \\
    &\leq 2 \| f(H_1) (\mo_\pm - \idop) \rVert \, \lVert  A f(H_1)^{-1} \rVert 
  \end{aligned}
 \end{equation*}
 which is finite by Prop.~\ref{prop:fequiv} and our assumption on $A f(H_1)^{-1}$.
\end{proof}

This theoretical result is of interest in quantum physics in the following situation \cite{BCL:backflow}. 

\begin{example}[Quantum mechanical backflow bounds]
In Example~\ref{example:laplace}, set $\ell=2$, $n=1$. Let $(X\psi)(x)=x\psi(x)$, write $P=-i\frac{d}{dx}$, and let $E$ be the spectral projector of $P$ for the interval $[0,\infty)$.  Pick a nonnegative Schwartz class function $g$ and set $J(g):=\frac{1}{2}(P g(X)+g(X) P)$ (the ``averaged probability flux operator'').  Then $E \mo_\pm^\ast J(g) \mo_\pm E$ is bounded below, but unbounded above.
\end{example}

\begin{proof}
It follows from elementary arguments that $EJ(g)E$ is bounded below, but unbounded above \cite[Theorem~1]{BCL:backflow}. 
With $f(\lambda):=(1+\lambda^2)^{1/4}$, we know from Example~\ref{example:laplace} that $H_0\sim_f H_1$. Also, writing $J(g) = g'(X) + 2 g(X)P$, it is clear that $J(g)f(H_0)^{-1}$ is bounded, hence also  $J(g)f(H_1)^{-1}$ by Eq.~\eqref{eq:abquot}. 
Further, since $H_0$ and $H_1$ are of strict high-energy order $\frac{1}{2}$, they are elements of $\aac(\Hil)$.
Thus, Theorem~\ref{thm:scatterineq} with $A=J(g)$ and $B=E$ shows that $E \mo_\pm^\ast J(g) \mo_\pm E$ is bounded below but unbounded above. 
\end{proof}

We have thus recovered our results on backflow bounds in \cite{BCL:backflow} for a slightly different class of potentials.  However, our present methods should allow to generalize the analysis of ``quantum inequalities'' in scattering situations to backflow of particles with inner degrees of freedom, as well as to a much larger class of semibounded operators relevant in quantum mechanics, e.g., those established in \cite{EvesonFewsterVerch:2003}.

\appendix

\section{\sectioncase{The Privalov-Korn theorem}}\label{app:Privalov}

We note the following variant, adapted to our purposes, of the Privalov-Korn theorem on singular Cauchy integrals (also known as the Plemlj-Privalov theorem or just as Privalov's theorem). For the convenience of the reader, we also sketch its proof.
\begin{lemma}\label{lemma:Privalovop}
 Let $\Banalt$ be a Banach space, let $a<b\in\rbb$, and let $B:[a,b]\to \Banalt$. Suppose there exists a constant  $\theta\in(0,1)$ such that
 \begin{equation*}
  \sup_{\lambda\in[a,b]}    \gnorm{ B(\lambda) }{\Banalt} 
  + \sup_{\lambda\neq\lambda'\in[a,b]} \Big\lvert \frac{b-a}{\lambda-\lambda'}\Big\rvert^{\theta}  \gnorm{ B(\lambda) - B(\lambda') }{\Banalt}    =:c <\infty.
 \end{equation*}
 Then, the $\Banalt$-valued function (with the integral defined in the weak sense)
 \begin{equation}\label{eq:cdef}
   C(\zeta) := \int_a^b \frac{B(\lambda) d\lambda}{\lambda-\zeta}, \quad \zeta\in \cbb\backslash\rbb,
 \end{equation}
 has limits $C(\lambda\pm i0) := \operatorname*{w-lim}_{ \epsilon \downarrow 0} C(\lambda \pm i \epsilon)$  for $a' \leq \lambda \leq b'$, where $a':=(1-\delta) a + \delta b$, $b' := \delta a + (1-\delta)b$, and $0 < \delta < \frac{1}{2}$ is some fixed number; it is locally H\"older continuous on $[a',b'] \pm i[0,\infty)$, and there is a constant $k_{\theta,\delta}>0$ (depending only on $\theta,\delta$, not on $a,b,c,B,\lambda,\Banalt$) such that
 \begin{equation}\label{eq:cest}
 \sup_{\lambda\in[a',b']}   \gnorm{ C(\lambda\pm i0) }{\Banalt} 
 +\sup_{\lambda\neq\lambda'\in[a',b']}   \Big\lvert \frac{b-a}{\lambda-\lambda'}\Big\rvert^{\theta}  \gnorm{ C(\lambda\pm i0) - C(\lambda'\pm i0) }{\Banalt}  \leq 
   k_{\theta,\delta} c.
 \end{equation}
\end{lemma}

We will often apply the theorem to $\Banalt=\bops(\Ban,\Band)$ with a Banach space $\Ban$, where it then also holds with respect to the weak operator topology.

\begin{proof}
In the case $a=0$, $b=1$, $\Ban=\cbb$, we obtain the existence of boundary values and the estimate \eqref{eq:cest} with standard arguments, e.g., as in \cite[Ch.~2, \S 19]{Musk:singular}, noting that the constants in the relevant estimates there depend linearly on $c$, but not directly on $B$. To show H\"older continuity in the complex domain $[a',b'] \pm i[0,\infty)$, we choose a smooth function $\chi:\Rl\to [0,1]$ such that $\chi(x)=1$ in a neighbourhood of $[a',b']$  and $\chi(x)=0$ in a neighbourhood of $(-\infty,a] \cup [b,\infty)$, and set $B'(\lambda)=\chi(\lambda)B(\lambda)$, $C'(\zeta) = \int B'(\lambda) (\lambda-z)^{-1} d\lambda$. Then $(C-C')(\zeta)$ is analytic in $\mathbb{H}_\pm$ and near $[a',b']$, in particular H\"older continuous. Also, $C'(\lambda\pm i0)$ exists for all $\lambda\in\Rl$ by \cite{Musk:singular}, and decays like $O(|\lambda|^{-1})$ for large $|\lambda|$. By a conformal transformation of $\overline{\mathbb{H}}_\pm$ to the closed unit circle, \cite[Supplement to Ch.~IV, \S 14]{CourantHilbert:mmp2} shows that $C'(z)$ is locally H\"older continuous with exponent $\theta$ in all of $\overline{\mathbb{H}}_\pm$. Again, one can choose a H\"older constant that depends linearly on $c$ only, not on further details of $B$.

For general $\Banalt$, choose $\varphi\in\Banalt^\ast$ and apply the previous result to the $\cbb$-valued function $\lambda \mapsto \varphi(B(\lambda))$; the resulting integral \eqref{eq:cdef} is of the form $\varphi(C(\lambda\pm i 0))$ with $C(\lambda\pm i 0)\in\Banalt$ by linearity in $\varphi$ and uniformity of the estimate \eqref{eq:cest} in $\gnorm{\varphi}{\Banalt^\ast}$. In particular, \eqref{eq:cest} follows for $C$. Similarly, by the uniformity of H\"older estimates mentioned above, H\"older continuity with respect to $\|\cdotarg\|_\Banalt$ follows for complex arguments.                    

Finally, for general $a<b\in \rbb$, consider $\hat B(\lambda) := B(a + \lambda(b-a))$, defined for $\lambda \in [0,1]$, and apply the previous results to $\hat B$.
\end{proof}

\section{\sectioncase{A lemma about differences of operators}} \label{app:difference}

We require estimates of differences $h(A)-h(B)$ where $A,B$ are selfadjoint unbounded operators, and $h$ is a certain function. The following is a special case of results by Birman, Solomyak and others; see \cite{Birman:doubleint} for a review. 

\begin{lemma}\label{lemma:fractional}
 Let $A,B$ be two selfadjoint operators on a common dense domain in a Hilbert space $\Hil$, such that $B-A$ is bounded. Let $h:\rbb\to\rbb$ be differentiable such that $h'\in L^p(\rbb)$ for some $p<\infty$, and suppose that $h'$ is (globally) H\"older continuous with some H\"older exponent $\epsilon>0$. Then $h(B)-h(A)$ is bounded. In particular, for any $\beta\in(0,1)$,
 \begin{equation}\label{eq:badiff}
    \lVert  (1+B^2)^{\beta/2} - (1+A^2)^{\beta/2} \rVert < \infty. 
 \end{equation}
 \end{lemma}

\begin{proof}
 The stated conditions on $h$ imply that, in the notation of \cite{Birman:doubleint}, the function $\phi_h(\mu,\lambda)=(h(\mu)-h(\lambda))/(\mu-\lambda)$ falls into the class $\mathfrak{M}$ \cite[Theorem~8.4]{Birman:doubleint}. Therefore, \cite[Theorem~8.1]{Birman:doubleint} is applicable with $\mathfrak{S}=\boundedops$, yielding the relation $h(B)-h(A)=Z_h^{A,B}(B-A)$ with a continuous map $Z_h^{A,B}:\boundedops\to\boundedops$.
 
 In particular, these conditions are fulfilled for $h(x)=(1+x^2)^{\beta/2}$, since  $h'(x)=O(|x|^{\beta-1})$ for large $|x|$, and $h''$ is bounded.
\end{proof}

As a consequence of Lemma~\ref{lemma:fractional}, also 
\begin{equation}\label{eq:abquot}
\lVert (1+B^2)^{-\beta/2}(1+A^2)^{\beta/2} \rVert < \infty
\end{equation}
by multiplying the bounded operator in \eqref{eq:badiff} with $(1+B^2)^{-\beta/2}$ from the left.

\section{\sectioncase{Some kernels of bounded operators}}\label{app:boundedkernel}

For our purposes, we need norm estimates of certain (singular) integral operators, which we collect here.

\begin{lemma} \label{lem:powerkernel}
 Let $0 < \gamma < \frac{1}{2}$. Then the (distributional) kernels
 \begin{equation*}
    K_1(\lambda,\mu) = \frac{(\lambda/\mu)^\gamma}{\lambda - \mu \pm i 0},
\qquad
    K_2(\lambda,\mu) = \frac{(\lambda/\mu)^\gamma}{\lambda + \mu}
 \end{equation*}
 induce bounded operators $T_1$, $T_2$ on $L^2(\rbb_+)$.
\end{lemma}
\begin{proof}
  Consider the unitary $U: L^2(\rbb_+) \to L^2(\rbb)$, $(Uf)(x) = e^{x/2} f(e^x)$. The operator $\hat T_1:=U T_1 U^\ast$ has the kernel
  \begin{equation}
     \hat K_1 (x,y) = e^{x/2} K_1(e^x,e^y) e^{y/2} = \frac{2 e^{\gamma(x-y)}}{\sinh(\frac{x-y}{2}\pm i0)}.
  \end{equation}
  This is a kernel of convolution type, hence we only need to show that the Fourier transform (in the sense of distributions)  of
  \begin{equation}
     f_1(z) = \frac{2e^{\gamma z}}{\sinh(z/ 2 \pm i0)}
  \end{equation}
is a bounded function. This can be extracted from the literature \cite[Sec.~17.23, formula 20--21]{GraRyd:table}, or obtained by comparison with a kernel with the same residue but simpler Fourier transform, e.g., $g(z)=4i(z+i)^{-1}(z \pm i0)^{-1}$, as $g-f$ is analytic and $L^1$. --- Likewise for $K_2$, it suffices to show that 
  \begin{equation*}
     f_2(z) = \frac{2e^{\gamma z}}{\cosh(z/ 2)}
  \end{equation*}
  has a bounded Fourier transform, which is clear since $f_2\in L^1(\Rl)$.
\end{proof}

\sloppy

\begin{acknowledgements} 
D.C.~is supported by the Deutsche Forschungsgemeinschaft (DFG) within the Emmy Noether grant CA1850/1-1.
H.B.~gratefully acknowledges a London Mathematical Society research in pairs grant, Ref. 41727.
\end{acknowledgements}

%% file: high-energy-wave_arxiv.bbl
\begin{thebibliography}{GKM02}
\expandafter\ifx\csname url\endcsname\relax
  \def\url#1{\texttt{#1}}\fi
\expandafter\ifx\csname doi\endcsname\relax
  \def\doi#1{\burlalt{doi:#1}{http://dx.doi.org/#1}}\fi
\expandafter\ifx\csname urlprefix\endcsname\relax\def\urlprefix{URL }\fi
\expandafter\ifx\csname href\endcsname\relax
  \def\href#1#2{#2}\fi
\expandafter\ifx\csname burlalt\endcsname\relax
  \def\burlalt#1#2{\href{#2}{#1}}\fi

\bibitem[BA11]{BenArtzi:smooth}
M.~Ben-Artzi.
\newblock Smooth spectral calculus.
\newblock In M.~Demuth, B.-W. Schulze, and I.~Witt, editors, {\em Partial
  Differential Equations and Spectral Theory}, page 119–182. Springer, Basel,
  2011.
\newblock \doi{10.1007/978-3-0348-0024-2_3}.

\bibitem[BCL17]{BCL:backflow}
H.~Bostelmann, D.~Cadamuro, and G.~Lechner.
\newblock Quantum backflow and scattering.
\newblock {\em Phys. Rev. A}, 96:012112, Jul 2017.
\newblock \doi{10.1103/PhysRevA.96.012112}.

\bibitem[BS03]{Birman:doubleint}
M.~S. Birman and M.~Solomyak.
\newblock Double operator integrals in a {Hilbert} space.
\newblock {\em Integral Equations and Operator Theory}, 47(2):131–168, Oct
  2003.
\newblock \doi{10.1007/s00020-003-1157-8}.

\bibitem[BW09]{BallesterosWeder:2009}
M.~Ballesteros and R.~Weder.
\newblock {High-velocity estimates for the scattering operator and
  Aharonov-Bohm effect in three dimensions}.
\newblock {\em Commun. Math. Phys.}, 285(1):345–398, 2009.

\bibitem[CH62]{CourantHilbert:mmp2}
R.~Courant and D.~Hilbert.
\newblock {\em Methods of Mathematical Physics}, volume {II}: Partial
  Differential Equations.
\newblock Interscience Publishers, New York, 1962.

\bibitem[CJLS16]{CJLS:matrixvalued}
G.~Cox, C.~K. R.~T. Jones, Y.~Latushkin, and A.~Sukhtayev.
\newblock The {Morse} and {Maslov} indices for multidimensional {Schrödinger}
  operators with matrix-valued potentials.
\newblock {\em Trans. Amer. Math. Soc.}, 368:8145–8207, 2016,
  \burlalt{1408.1103}{http://arxiv.org/abs/1408.1103}.

\bibitem[DG10]{DurhuusGayral:2010}
B.~Durhuus and V.~Gayral.
\newblock The scattering problem for a noncommutative nonlinear {Schrödinger}
  equation.
\newblock {\em SIGMA}, 6:046, 2010,
  \burlalt{0903.1493}{http://arxiv.org/abs/0903.1493}.
\newblock \doi{10.3842/SIGMA.2010.046}.

\bibitem[DT79]{DT:scattering}
P.~Deift and E.~Trubowitz.
\newblock Inverse scattering on the line.
\newblock {\em Comm. Pure Appl. Math.}, 32(2):121–251, 1979.
\newblock \doi{10.1002/cpa.3160320202}.

\bibitem[EFV05]{EvesonFewsterVerch:2003}
S.~P. Eveson, C.~J. Fewster, and R.~Verch.
\newblock Quantum inequalities in quantum mechanics.
\newblock {\em Annales H. Poincaré}, 6(1):1–30, 2005.
\newblock \doi{10.1007/s00023-005-0197-9}.

\bibitem[ES80]{EnssSimon:crosssection}
V.~Enss and B.~Simon.
\newblock Finite total cross-sections in nonrelativistic quantum mechanics.
\newblock {\em Commun. Math. Phys.}, 76(2):177–209, Jun 1980.
\newblock \doi{10.1007/BF01212825}.

\bibitem[FLS07]{FLS:boundstates}
R.~L. Frank, E.~H. Lieb, and R.~Seiringer.
\newblock Number of bound states of {Schrödinger} operators with matrix-valued
  potentials.
\newblock {\em Lett. Math. Phys.}, 82:107–116, 2007.
\newblock \doi{10.1007/s11005-007-0211-x}.

\bibitem[GKM02]{GKM:matrixvalued}
F.~Gesztesy, A.~Kiselev, and K.~A. Makarov.
\newblock Uniqueness results for matrix-valued {Schrödinger}, {Jacobi}, and
  {Dirac}-type operators.
\newblock {\em Math. Nachr.}, 239-240(1):103–145, 2002.
\newblock \doi{10.1002/1522-2616(200206)239:1<103::AID-MANA103>3.0.CO;2-F}.

\bibitem[GR07]{GraRyd:table}
I.~S. Gradshteyn and I.~M. Ryzhik.
\newblock {\em Table of Integrals, Series, and Products}.
\newblock Academic Press, 7th edition, 2007.

\bibitem[Jen80]{Jensen:crosssection}
A.~Jensen.
\newblock The scattering cross section and its {Born} approximation at high
  energies.
\newblock {\em Helvetica Physica Acta}, 53:398–403, 1980.

\bibitem[KK71]{KatoKuroda:scattering}
T.~Kato and S.~Kuroda.
\newblock The abstract theory of scattering.
\newblock {\em Rocky Mountain J. Math.}, 1(1):127–172, 03 1971.
\newblock \doi{10.1216/RMJ-1971-1-1-127}.

\bibitem[KR08]{KamRou:matrixresolvent}
C.~F. Kammerer and V.~Rousse.
\newblock Resolvent estimates and matrix-valued {Schrödinger} operator with
  eigenvalue crossings; application to {Strichartz} estimates.
\newblock {\em Communications in Partial Differential Equations},
  33(1):19–44, 2008.
\newblock \doi{10.1080/03605300701454925}.

\bibitem[LV15]{LechnerVerch:2013}
G.~Lechner and R.~Verch.
\newblock {Linear hyperbolic PDEs with noncommutative time}.
\newblock {\em J. Noncommutative Geometry}, 9(3):999–1040, 2015.
\newblock \doi{10.4171/JNCG/214}.

\bibitem[Mus58]{Musk:singular}
N.~Muskhelishvili.
\newblock {\em Singular integral equations}.
\newblock Noordhoff, Groningen, 1958.

\bibitem[Pis86]{Pisier:factorization}
G.~Pisier.
\newblock {\em Factorization of Linear Operators and Geometry of {Banach}
  Spaces}.
\newblock Number~60 in Conference Board of the Mathematical Sciences. American
  Mathematical Society, 1986.

\bibitem[RS79]{ReedSimon:1979}
M.~Reed and B.~Simon.
\newblock {\em {Methods of Modern Mathematical Physics III - Scattering
  Theory}}.
\newblock Academic Press, 1979.

\bibitem[Sim95]{Simon:1994}
B.~Simon.
\newblock Spectral analysis of rank one perturbations and applications.
\newblock In {\em Mathematical Quantum Theory II: Schrödinger Operators},
  volume~8 of {\em CRM proceedings \& lecture notes}, page 109–149. American
  Mathematical Society, 1995.

\bibitem[Sim15]{Simon:analysis4}
B.~Simon.
\newblock {\em A Comprehensive Course in Analysis}, volume 4: Operator theory.
\newblock AMS, Providence (Rhode Island), 2015.

\bibitem[SY86]{SobolevYafaev:qcl}
A.~V. Sobolev and D.~Yafaev.
\newblock On the quasi-classical limit of the total scattering cross-section in
  nonrelativistic quantum mechanics.
\newblock {\em Annales de l'I.H.P. Physique théorique}, 44(2):195–210, 1986.
\newblock \urlprefix\url{http://www.numdam.org/item/AIHPA_1986__44_2_195_0}.

\bibitem[Yaf92]{Yafaev:general}
D.~R. Yafaev.
\newblock {\em Mathematical Scattering Theory: General Theory}, volume 105 of
  {\em Translations of Mathematical Monographs}.
\newblock American Mathematical Society, 1992.

\bibitem[Yaf10]{Yafaev:analytic}
D.~R. Yafaev.
\newblock {\em Mathematical Scattering Theory: Analytic Theory}, volume 158 of
  {\em Mathematical Surveys and Monographs}.
\newblock American Mathematical Society, 2010.

\end{thebibliography}
